\documentclass[12pt,leqno]{article}
\usepackage{graphicx}
\usepackage{amsmath,amsthm}
\usepackage[round]{natbib}

\usepackage{rotating}
\usepackage{enumerate}
\usepackage{mathrsfs}
\usepackage{hyperref}
\hypersetup{
urlcolor=blue,
colorlinks=true,
citecolor=black,
linkcolor=black,
}

\usepackage{xcolor,colortbl}
\definecolor{green}{rgb}{0.1,0.1,0.1}

\usepackage{amssymb}
\usepackage{multirow,array,epsfig}
\usepackage{verbatim}

\newtheorem{theorem}{Theorem}
\newtheorem{definition}{Definition}[section]
\newtheorem{assumption}{Assumption}
\newtheorem{lemma}{Lemma}
\newtheorem{remark}{Remark}[section]

\newcommand\bA{{\bf A}}
\newcommand \mcA{\mathcal{A}}

\newcommand\bB{{\bf B}}

\newcommand\mcD{{\mathcal D}}
\newcommand\bE{{\bf E}}
\newcommand\bF{{\bf F}}
\newcommand \mscF{\mathscr{F}}
\newcommand\bI{{\bf I}}

\newcommand\mcH{{\mathcal H}}

\newcommand\bR{{\bf R}}

\newcommand\mbR{{\mathbb R}}

\newcommand\mcT{{\mathcal T}}

\newcommand\bW{{\bf W}}

\newcommand\bX{{\bf X}}

\newcommand\eg{\varepsilon}

\DeclareMathOperator{\Cov}{Cov}
\DeclareMathOperator{\logit}{logit}

\DeclareMathOperator{\vc}{vec}
\DeclareMathOperator{\argmin}{argmin}
\DeclareMathOperator{\trace}{trace}

\newcommand{\norm}[1]{\left\lVert{#1}\right\rVert}

\newcommand{\edits}[1]{{\color{black}{#1}}}

\usepackage{authblk}

\title{The Function-on-Scalar LASSO \\ with Applications to Longitudinal GWAS}
\author[1]{Rina Foygel Barber}
\author[2]{Matthew Reimherr \thanks{Corresponding Author: Department of Statistics, Pennsylvania State University, University Park, PA, 16802}}
\author[2]{Thomas Schill }
\affil[1]{University of Chicago}
\affil[2]{Pennsylvania State University}
\date{}

\begin{document}
\maketitle

\begin{abstract}
We present a new methodology for simultaneous variable selection and parameter estimation in function-on-scalar regression with an ultra-high dimensional predictor vector.  We extend the LASSO to functional data in both the \textit{dense} functional setting and the \textit{sparse} functional setting.  We provide theoretical guarantees which allow for an exponential number of predictor variables.  Simulations are carried out which illustrate the methodology and compare the sparse/functional methods.  Using the Framingham Heart Study, we demonstrate how our tools can be used in genome-wide association studies, finding a number of genetic mutations which affect blood pressure and are therefore important for cardiovascular health. 
\end{abstract}


\section{Introduction}
Over the last several decades, technological advances have supported and necessitated the rapid growth of high-dimensional statistical methods.  One the most important applications for these methods are genome-wide association studies (GWAS).  In these studies, researchers search through hundreds of thousands or millions of genetic mutations known as SNPs, single nucleotide polymorphisms, finding those which significantly impact an outcome or phenotype of interest, e.g.\ blood pressure, diabetes, asthma, etc.  GWAS have been hugely successful at finding gene/disease associations as evidenced by the massive repositories of genetic studies and findings, such as dbGaP (http://www.ncbi.nlm.nih.gov/gap).  Next-generation or high-throughput sequencing technologies are capable of sequencing entire genomes, producing ever-larger genetic datasets to explore.  Thus, to further our understanding of the genetic architecture of complex human diseases, there is a substantial and continuing need for powerful, high-dimensional association techniques.  

The vast majority of GWAS are cross-sectional, examining only one point in time.  The aim of this paper is to present a new framework which combines functional data analysis, FDA, and machine leaning for finding and estimating significant effects on longitudinally measured outcomes.  We refer to this methodology as \textit{Function-on-Scalar LASSO}, or FS-LASSO.  \textcolor{black}{As in \citet{reiss:etal:2010}, we use the term \textit{Function-on-Scalar} to help distinguish our setting from the wide array of regression problems now in existence.}  
{\color{black} In GWAS, the ``function'' is the time-varying phenotype of the individual, and it is being regressed on the ``scalar'',  which is the individual's constant (non-time-varying) genotype.}

Our aim is to simultaneously exploit the sparse effect of the SNPs and the smooth nature of the longitudinal outcomes.  While we are strongly motivated by genetic studies, our methods are general and allow for any setting with a longitudinal/functional outcome and high-dimensional scalar predictors.  Our primary goal is to select and estimate the effect of predictors in the following functional linear model
\begin{align}
Y_n(t) = \mu(t) + \sum_{i=1}^I X_{ni} \beta_i(t) + \eg_n(t). \label{e:model}
\end{align}
Here $Y_n(t)$ is the value of a quantitative outcome for subject $n\in \{1,\dots,N\}$, at time $t \in \mcT \subset \mbR.$  The scalars $X_{ni}$ are real-valued, though in our application they take values in $\{0,1,2\}$ indicating the minor allele count for the SNP.  The number of predictors, $I$, is allowed to be much greater than the sample size $N$.  This is known as a ``scalar predictor/functional response'' model.  Functional data analysis (FDA) now consists of two main branches: (1) sparse FDA, where the outcomes, $Y_n$, are observed at a relatively small number of time points and are contaminated with noise, and (2) dense FDA, where outcomes are observed at a large number of time points (possibly with a small level of noise). 
\textcolor{black}{In practice, some data clearly fall into one of these two categories, say with hundreds or thousands of observations per subject, or with only a handful.  There are also many settings where it is not clear which scenario one is in.  While establishing a clear cutoff still a relatively open problem, we refer the interested reader to \citet{li:hsing:2010} and \citet{zhang:wang:2015} who suggest the line is when the number of points per curve is greater/less than the quad-root of the sample size.  We emphasize that this is just a basic rule of thumb, and care must be taken with each application.}
We present our methodology for each scenario and we later explore how their performances differ via simulations.  
\textcolor{black}{In the functional case, as we will see, our methods and theory will actually include a wide range of settings beyond just \eqref{e:model}}.

At the heart of our methodology is the now classic idea of combining basis expansions for estimating the functions $\beta_i$, with proper penalties which induce sparse estimates.  While the methods and theory are new, we can phrase the problem as a type of group LASSO, which allows us to utilize existing computational tools.

The contributions of this paper include the following.  First, we provide new asymptotic results for the sparse setting.  Our results can be viewed as an extension of the varying coefficient methods discussed in \citet{wei:huang:li:2011}.  In particular, we present a new restricted eigenvalue result which is one of the cornerstones for developing convergence rates in high dimensions, as well as the accompanying asymptotic convergence rates for our estimates.  Second, we provide a new methodology and accompanying asymptotic theory for a broad class of dense settings, which, to the best of our knowledge, has not been explored before.  Interestingly, we show that the FS-LASSO applied to outcomes in any separable Hilbert space achieves the exact same rates of convergence as in the scalar setting.  This could have broad implications for not only traditional functional outcomes, but also spatial processes, functional panels, and imaging data such as fMRI.  Finally, we demonstrate via simulations some surprising results concerning the choice between the sparse and dense tool sets.  In particular, we find that a dense approach is comparable to the sparse in terms of variable selection even in traditionally sparse settings.  Furthermore, the dense methods can be carried out at a fraction of the computational cost (both in terms of power and memory).  However, the sparse methods produce estimates which are more accurate in traditionally sparse settings.  This opens the doors to interesting two-stage procedures where variable screening is done using dense tools, and final estimation is done via sparse ones.

\paragraph{\color{black} Related literature} For foundations on functional data analysis we refer to \citet{ramsay:silverman:2005} and \citet{hkbook}, while an overview of machine learning can be found in \citet{HaTiFr:2001} and \citet{james:etal:2013}.  The literature on functional regression is now quite large, but we attempt to provide several key methodological papers, which help outline the field.
Functional regression methods for sparse data include the following.  \citet{hoover:etal:1998} examine spline based methods for estimating a simpler form of the functional predictor/functional response model. \citet{fan:zhang:2000} provide a two step method based on local polynomial smoothing to estimate a functional predictor/scalar response model.  \citet{yao:muller:wang:2005AS} present what has likely become the most common method based on scatter plot smoothing (i.e. local polynomial smoothing) and functional principal components for estimating a full functional predictor/functional response model.  \citet{zhu:etal:2012} extend local linear smoothing methods for estimating scalar predictor/multivariate functional response models.  Methods for high dense FDA include the following.  \citet{cardot:ferraty:sarda:2003} explore spline based methods for a functional predictor/scalar response model. \citet{kokoszka:maslova:s:z:2008} use a PCA based approach for estimating a full functional predictor/functional response model.  \citet{james:wang:zhu:2009} incorporate shrinkage methods in estimating a functional predictor/scalar response model (for the purposes of estimating the "zero" parts of the regression function).  \citet{reiss:etal:2010} examine a B-spline approach for scalar predictor/functional response models, while \citet{ReNi:2014} explore more direct least squares (of a functional norm) estimators for scalar predictor/functional response models.  However, all of these methods assume a fixed number of covariates.  
Related work which incorporates an increasing number of predictors mainly comes from the literature on varying coefficient models, for example \citet{wang:etal:2008}, \citet{zhao:xue:2010} and \citet{wei:huang:li:2011}.  The only related work in the FDA literature concerns scalar-on-function regression, the reverse of our focus.  
\textcolor{black}{In particular, \citet{MaKo:2011} combine basis expansions with a SCAD style penalty, but give no theoretical guarantees.   \citet{Lian:2013} also uses a SCAD penalty, but combined with FCPA instead of a general basis function.  Theory is provided, but only in the case where the number of predictors is fixed.  \citet{gertheiss:maity:staicu:2013} developed a procedure for a generalized scalar-on-function linear model again using basis expansions, but combined with a LASSO style penalty which simultaneously shrinks and smooths parameter estimates; no theoretical guarantees are given.  \citet{fan2014functional,fan:james:rad:2015} consider functional additive regression models with a large number of predictors.  Their methods allow for nonlinear relationships and theoretical properties are established.  Lastly, \cite{kong:2016partially} consider the problem of variable selection with a scalar outcome and scalar predictors, but with a small set of functional predictors that must be accounted for as well.  We emphasize that all of these methods are for the scalar-on-function case; for the function-on-scalar case the only work we are aware of is by \citet{chen:goldsmith:ogden:2016} who consider a basis expansion approach with a MCP style penalty and fixed number of covariates, but their method cannot be applied to settings where $I \gg N$. }

\paragraph{\color{black} Outline} 
The remainder of the paper is organized as follows.  In Section \ref{s:sparse} we provide a framework for scalar predictor/functional response regression with sparsely observed outcomes.  We combine basis expansions with a group LASSO style penalty to carry out variable selection and parameter estimation simultaneously.  In Section \ref{s:func} we present the high dense setting.  We provide a very general framework where the response functions are allowed to take values from a separable Hilbert space.  This allows the functions to be from say $L^2(\mcT)$, which is commonly used, as well as more general spaces such as Sobelev spaces or product spaces of function spaces (i.e. vectors of functions).  In Section \ref{s:comp} we provide computational details, including a pre-screening rule which allows one to reduce $I$ substantially before fitting the FS-LASSO, and details for making use of established computational machinery for the group LASSO. We present a simulation study in Section \ref{s:sims} where we explore our procedure in terms of variable selection, parameter estimation, and computational time.  In Section \ref{s:FHS} we apply our methods to the Framingham Heart Study (FHS), to identify genetic markers that influence blood pressure, in hopes of gaining a better understanding of cardiovascular disease.  Genome--wide SNP data and phenotype information were downloaded from dbGaP (http://www.ncbi.nlm.nih.gov/gap) study accession phs000007.v25.p9.
Concluding remarks are given in Section \ref{s:conc}, while all theoretical results are proven in the Appendix.

\section{The Sparse Setting}\label{s:sparse}
Sparse FDA occurs quite often in longitudinal studies where one has a relatively small number of observed time points per subject or unit.  In this case, it is common to work with the raw measurements as opposed to the functional embeddings discussed in Section \ref{s:func}.  A brief introduction to sparse FDA can be found in \citet{muller:2008}.  

We begin with the underlying model.  
\begin{assumption} \label{a:sparse}
Assume, for $ 1 \leq n \leq N$ and $1 \leq m \leq M_n \leq M < \infty$, that
\[
Y_{nm} = \sum_{i=1}^I X_{ni} \beta_i^\star(t_{nm}) + \eg_n(t_{nm}), \ \ \ t_{nm} \in \mcT .
\]
The error processes $\{\eg_n(t): t\in \mcT\}$ are iid Gaussian and we define the 
 $M_n \times M_n$ covariance matrix
$\Sigma_n: = \Cov\big((\eg_n(t_{n1}), \dots, \eg_n(t_{nM_n}))\big)$.  The design matrix, $\bX$, is deterministic.
\end{assumption}
\textcolor{black}{Throughout, we will use a $\star$ to denote true parameter values.}   For the moment, we do not include assumptions on $\{t_{nm}\}$ or coefficient functions $\beta_i^\star$.  Note that the Gaussian assumption is not crucial---what is really needed is that the processes have subgaussian tails, though we make the Gaussian assumption to simplify the arguments.  

Let $\{e_j(\cdot)\}$ be a basis in $L_2(\mcT)$.  Approximating the functions $\beta_i^\star$ using this basis, we have that
\[
Y_{nm} = \sum_{i=1}^I \sum_{j=1}^J X_{ni} B_{ij}^\star e_j(t_{nm}) + T_{nm} + \eg_n(t_{nm}),
\]
where $T_{nm}$ is the truncation error obtained after cutting off the basis expansion at $J$.  The following notation will be used repeatedly throughout this section and in the Appendix:
 $$T = (T_{11}, T_{12}, \dots, T_{NM_N})^\top \in \mbR^{\sum_n M_n},$$
$$E_{nm} = (e_1(t_{nm}),\dots,e_J(t_{nm}))^\top \in \mbR^{J},$$ 
$$\bE_n = \left(E_{n1}, \dots E_{nM_n}  \right)^\top \in \mbR^{M_n \times J},$$
$$ \bF = \frac{1}{N} \sum_{n=1}^N\bE_n^\top \bE_n \in \mbR^{J \times J}.$$
For a matrix $\bB \in \mbR^{I\times J}$, define the \textit{sparse target function} as
\[L(\bB)=\frac{1}{2}\sum_{n=1}^N\sum_{m=1}^{M_n}\left(Y_{nm}-\sum_{i=1}^I\sum_{j=1}^JX_{ni}B_{ij}e_j(t_{nm})\right)^2 +\lambda\norm{\bB}_{\ell_1/\ell_2}\;,\]
where $\norm{\bB}_{\ell_1/\ell_2}=\sum_i \norm{B_{i*}}_2$ promotes row-wise sparsity (here $B_{i*}$ denotes the $i$th row of the matrix $\bB$).  The estimate $\widehat \bB$ is the minimizer of the above expression.  The estimated coefficient functions are then given by $\widehat \beta_i(\cdot)=\sum_{j=1}^J \widehat B_{ij}e_j(\cdot)$.
The target function can be rephrased so that traditional group LASSO machinery can be invoked.  Notice that
\[
\sum_{i=1}^I\sum_{j=1}^JX_{ni}B_{ij}e_j(t_{nm}) = X_n^\top \bB E_{nm}
= E_{nm}^\top \bB^\top  X_n
 = (X_n^{\top} \otimes E_{nm}^\top ) \vc(\bB^\top).
\]
 One can then stack the $(X_n^{\top} \otimes E_{nm}^\top ) $ vectors into a matrix, $\bA$,
\[
\bA = \left(\begin{matrix}
 X_1^{\top} \otimes E_{11}^\top \\
X_1^{\top} \otimes E_{12}^\top  \\
\vdots\\
X_N^{\top} \otimes E_{NM_N}^\top  
\end{matrix} \right) = 
\left(\begin{matrix}
 X_1^{\top} \otimes \bE_{1} \\
X_2^{\top} \otimes \bE_{2}  \\
\vdots\\
X_N^{\top} \otimes \bE_{N}
\end{matrix} \right)\in\mbR^{(\sum_n M_n) \times IJ}.
\]
The target function can now be expressed as
\begin{equation}\label{eqn:grouplasso}
L(\bB) = \frac{1}{2} \| Y - \bA \vc(\bB^\top)\|_2^2 + \lambda \| \bB\|_{\ell_1/\ell_2},
\end{equation}
where
\[
Y = (Y_{11}, \dots, Y_{1M_1},\dots,Y_{N,M_N})^\top.
\]
Group LASSO computational tools can then be used to find $\widehat \bB$. 

Since the matrix $\bA$ will generally have more columns than rows, the linear system $Y\approx \bA\vc(\bB^\top)$ is underdetermined; in particular the (right) null space of $\bA$ is large. However, the grouped sparsity structure in $\bB$ allows us to resolve this difficulty. We first recall the notion of restricted eigenvalues \cite{bickel2009simultaneous}, used for sparse regression (without group structure): 
\begin{definition}[Restricted eigenvalue condition]\label{def:RE}
A matrix $\bA\in\mathbb{R}^{N\times I}$ satisfies the $\mathsf{RE}(I_0,\alpha)$ condition if, for all subsets $S\subset \{1,\dots, I\}$ with $|S|\leq I_0$,
\[\norm{\bA w}^2_2\geq \alpha N\norm{w}^2_2 \text{ for all }w\in\mathbb{R}^I\text{ with }\norm{w_{S^c}}_{\ell_1}\leq 3\norm{w_S}_{\ell_1}\;.\]
\end{definition}
\noindent The constant $3$ here is somewhat arbitrary, but it is standard in the literature and we use it here for convenience.
\citet{lounici2011oracle} extend this definition to the group-sparse setting:
\begin{definition}[Grouped restricted eigenvalue condition]\label{def:RE_group}
A matrix $\bA\in\mathbb{R}^{N\times IJ}$ satisfies the $\mathsf{RE}_{\mathsf{group}}(I_0,\alpha)$ condition if, for all subsets $S\subset \{1,\dots,I\}$ with $|S|\leq I_0$,
\[\norm{\bA  \vc(\bW^\top)}^2_2 \geq \alpha N\norm{\bW}^2_{\mathsf{F}}\text{ for all }\bW\in\mathbb{R}^{I\times J} \text{ with }\norm{\bW_{\!S^c}}_{\ell_1/\ell_2}\leq 3\norm{\bW_{\!S}}_{\ell_1/\ell_2}\;.\]
\end{definition}
\noindent Here $\bW_{\!S}$ denotes the submatrix of $\bW$ obtained by extracting the rows indexed by $S$, while $\bW_{\!S^c}$ contains
only the rows in $S^c$, the complement of $S$.

\subsection{Bounding the error}

With these definitions in place, our first result is a modified version of \citet{lounici2011oracle}'s analysis of the group LASSO. It proves resulting bounds on the error $\widehat{\bB}-\bB^\star$, where $\bB^\star$ is the (truncated) true parameter matrix, as long as $\lambda$ is sufficiently large and $\bA$ satisfies the grouped restricted eigenvalue condition.
\begin{theorem}\label{thm:main_sparse}
Suppose that Assumption \ref{a:sparse} holds and that for some $\delta \in (0,1)$
\begin{equation}\label{eqn:lambda_bound}
\lambda \geq 2\sqrt{N\norm{\bF}_\mathsf{op}}\left(\norm{T}_2 + \sqrt{\max_n\norm{\Sigma_n}_\mathsf{op}\left(2J+3\log(2I/\delta)\right)}\right)\;.
\end{equation}
Assume also that $\bB^\star$ has at most $I_0$ nonzero rows, and that $\bA$ satisfies the $\mathsf{RE}_{\mathsf{group}}(I_0,\alpha)$ condition for some $\alpha>0$.
Then with probability at least $1 - \delta$, any minimizer $\widehat{\bB}$ of \eqref{eqn:grouplasso}  satisfies
\[\norm{\widehat{\bB}-\bB^\star}_{\mathsf{F}}\leq \frac{3\lambda\sqrt{I_0}}{\alpha N}\text{ and }\norm{\widehat{\bB}-\bB^\star}_{\ell_1/\ell_2}\leq \frac{12\lambda I_0}{\alpha N}\;.\]
\end{theorem}
 \edits{\begin{remark}
 Results for sparse regression in a non-grouped setting commonly give error bounds in both an $\ell_2$ norm and an $\ell_1$ norm; the former typically gives a more favorable scaling of sample size with respect to sparsity, while the latter ensures that for a larger sample size the error itself is nearly sparse, i.e.~does not have a ``long tail'' of small errors on many coordinates. Similarly, here we give results in both the Frobenius norm (with more favorable sample size scaling) and the $\ell_1/\ell_2$ norm.
 \end{remark}
\begin{remark}
As is standard in the LASSO literature, the consistency results given in this theorem can yield a selection consistency result as well: if we assume that $\min_{i\in I_0}\norm{\bB^\star_i}_2>\frac{6\lambda\sqrt{I_0}}{\alpha N}$, then the bound on $\norm{\widehat{\bB}-\bB^\star}_{\mathsf{F}}$ implies that the following estimated support,
{\color{black}
\[\widehat S = \left\{i\in\{1,\dots, I\} : \norm{\widehat\bB_{i_*}}_2>\frac{3\lambda\sqrt{I_0}}{\alpha N}\right\}\;,\]
is correct (i.e.~$\widehat S$ is equal to the row support of $\bB^\star$) with high probability.}
\end{remark}}

As is widely appreciated by researchers working on such results, Theorem \ref{thm:main_sparse} essentially hinges on concentration inequalities and restricted eigenvalue type conditions.  Neither of these tools can be readily applied to our setting.  Thus, much of our theoretical work is focussed on showing how to extend these ideas to the functional setting.  Below we discuss the restricted eigenvalue condition at length, while discussion of the concentration inequalities involved can be found at the end of Appendix \ref{a:proof_sparse}.
Concentration inequalities allow us to control the stochastic error in the model and largely dictate the rates of convergence, while the restricted eigenvalue conditions allow us to bound the error $\norm{\widehat{\bB}-\bB^\star}_{\mathsf{F}}$ from above by a term involving $\| \bA \vc\big((\widehat{\bB}-\bB^\star)^\top\big)\|_{2}$, and therefore to obtain convergence rates for the parameters of interest. 

\textcolor{black}{Note that the phrasing of Theorem \ref{thm:main_sparse} allows for the possibility of multiple minimizers of \eqref{eqn:grouplasso}, guaranteeing only that any minimizer must achieve the specified converge rates.  However, in practice it rare is for LASSO style estimates to not be unique.  We direct the interested reader to \citet{tibshirani2013lasso} for a thorough discussion of this issue.} \\

\noindent{\bf Example:} 
We next give an example to illustrate the type of scaling that is obtained in Theorem~\ref{thm:main_sparse}. 
Let $W^{\tau,2}[0,1]$ denote the Sobelev space consisting of functions over $[0,1]$, which have square integrable derivatives of up to order $\tau$.  Suppose that the coefficient functions, $\beta_i$, take values in a finite ball in $W^{\tau,2}$.  Let the basis $e_1(\cdot), e_2(\cdot), \dots$ denote the Fourier basis and assume that the time points $t_{nm}$ are iid $U[0,1]$.  Examining $\lambda$, there are now a number to terms which can be made more explicit.  First, since the basis is orthogonal we have $\bF  \approx \bI_{J \times J}$, so that $\| \bF \|_{\mathsf{op}} $ is behaves like a constant.  Next, examining the truncation term $T$, there are $\sum_n M_n \sim N$ coordinates, and each coordinate consists of a sum of $I_0$ functions each of which lies in a Sobelev ball, thus  we have that $\| T\| \approx  \sqrt {N I_0}  J^{-\tau}$.   Examining the second term in $\lambda$, we have that
\[
\sqrt{ \max_n\norm{\Sigma_n}_{\mathsf{op}}\left(2J+3\log(2I/\delta)\right)}
\sim \sqrt{J+\log(I)} \leq \sqrt{J \log(I)}. 
\]
We therefore want to choose $J$ to balance the two errors, which means taking $J$ such that
\[
\sqrt {N I_0}  J^{-\tau} = \sqrt{J \log(I)} \Longrightarrow J = \left( \frac{N I_0}{\log(I)} \right)^{\frac{1}{1+2\tau}}.
\]
Plugging this into our expression for the convergence rates,
and assuming that the restricted eigenvalue condition $\mathsf{RE}_{\mathsf{group}}(I_0,\alpha)$  holds for the matrix $\bA$ with some  $\alpha>0$ that we treat as a constant,
 we have that
\begin{align*}
 \| \Delta\|_{\mathsf{F}} & = O_P(1) \cdot \frac{N^{1/2}  \sqrt{J \log(I)} I_0^{1/2}}{N} 
 = O_P(1)\cdot  I_0 \left( \frac{ \log(I)}{ N I_0} \right) ^{\frac{\tau}{1+2 \tau}}, 
\end{align*}
where $\Delta:= \bB - \widehat \bB$.
We see that the standard nonparametric rate of convergence appears which relates the convergence rates to the smoothness of the underlying parameter functions.  This rate also applies to the difference $\|\widehat \beta - \beta\|_{L^2}$.  Notice that by Parceval's identity and the Triangle inequality
\[
\|\widehat \beta - \beta\|_{L^2} = \left( \sum_{i=1}^I \int (\widehat \beta_i(t) - \beta_i(t)^2 ) \ dt  \right)^{1/2}
\leq \|\Delta\|_{\mathsf{F}} + \|\Pi_J^\perp \beta\|_{L^2},
\] 
where $\Pi_J^\perp$ is  the projection onto the remaining basis functions $e_{J+1}(\cdot),\dots$.  By the assumed smoothness of the $\beta$, we have that
\[
\|\Pi_J^\perp \beta\|_{L^2} \sim \sqrt{I_0}J^{-\tau}
=  \sqrt{I_0}
 \left( \frac{N I_0}{\log(I)} \right)^{\frac{-\tau}{1+2\tau}} 
 \leq I_0 \left( \frac{ \log(I)}{ N I_0} \right) ^{\frac{\tau}{1+2 \tau}},
\]
Thus the rate for $\|\widehat \beta - \beta\|_{L^2}$ is the same as for $\| \Delta\|_{\mathsf{F}}$.

For the $\ell_1/\ell_2$ norm we then have
\[
\| \Delta\|_{\ell_1/\ell_2} = O_P(1)\cdot \frac{\lambda I_0}{N} = O_P(1)\cdot I_0^{3/2} \left( \frac{ \log(I)}{ N I_0} \right) ^{\frac{\tau}{1+2 \tau}}.
\]
By the same arguments as before, the above rate also applies to the sum of the normed differences between the functions, i.e. $\| \widehat \beta - \beta\|_{\ell_1}$.  As we will see, when $\tau \to \infty$, the rate above converges to the rate found for the dense setting.

\subsection{Restricted eigenvalue condition}

Our next result concerns the restricted eigenvalue condition. It proves that, if the covariates $X_n$ are drawn from a Gaussian or subgaussian distribution with well-conditioned covariance structure, then the grouped restricted eigenvalue condition will hold for $\bA$.
\begin{theorem}\label{t:RE_gaussian}
\edits{Suppose that $X_n \overset{iid}{\sim} N(0,\Sigma)$ with $\norm{\Sigma}\leq \nu^2$, or more generally, the $X_n$'s are iid 
where each $X_n$ has mean $0$, covariance $\Sigma$, and is $\nu$-subgaussian meaning that $\mathbb{E}[e^{v^\top X_n}]\leq e^{\nu^2\norm{v}^2_2/2}$ for any fixed vector $v$. Suppose that $\lambda_{\min}(\Sigma)>0$. Consider any fixed sequence of matrices $\bE_1 \in \mbR^{ M_1 \times J}$, \dots, $\bE_N\in \mbR^{ M_N \times J}$ satisfying
}
\begin{equation}\label{e:Econd}
\min_{w\in\bR^J\backslash\{0\}}\frac{\sum_n \norm{\bE_n w}_2}{N\norm{w}_2}\geq \gamma_0>0, \quad \max_{w\in\bR^J\backslash\{0\}}\sqrt{\frac{\sum_n \norm{\bE_n w}^2_2}{N\norm{w}^2_2}}\leq \gamma_1\;,
\end{equation}
  and define $\bA\in \mbR^{\sum_n M_n \times IJ}$ as before. 
Then there exist  $c_0,c_1,c_2>0$ depending only on $\nu,\lambda_{\min}(\Sigma),\gamma_0,\gamma_1$, such that, if 
\begin{equation}\label{eqn:N_RIP_thm}N \geq c_0\cdot  k \cdot \log(IJ)\cdot (J+\log(I))\;,\end{equation}
then with probability at least $1- e^{-c_1N}$, $\bA$ satisfies the $\mathsf{RE}_{\mathsf{group}}(k,c_2)$ condition.
\end{theorem}
\edits{
\begin{remark}
This theorem is a primary distinction between our work discussed thus far and that of \citet{wei:huang:li:2011}. Their work cites existing results which imply that the matrix $\bA$ behaves appropriately for each approximately-row-sparse $\bW$  in expectation, that is, $\mathbb{E}\big[\norm{\bA  \vc(\bW^\top)}^2_2\big]$ is lower-bounded for each $\bW$. In contrast, our theorem above proves the much stronger statement that $\norm{\bA  \vc(\bW^\top)}^2_2$ can be lower-bounded  simultaneously with high probability for all approximately-row-sparse $\bW$, which is necessary for estimation in high dimensions.\end{remark}}

\edits{\begin{remark}
The conditions~\eqref{e:Econd} on the basis matrices $\bE_n$ are of course dependent on the choice of basis; this type of basis-dependent
assumption is standard in the nonparametric literature. As a simple example of a choice of basis where~\eqref{e:Econd} is satisfied, suppose that there are $J$ evenly spaced time points and our basis is given by $J$ orthonormal functions (e.g.~a Fourier basis). Suppose measurements for each individual are taken at a random subset of $J'$ time points from the original $J$. Then each $\mathbf{E}_n$ is a $J'\times J$ submatrix chosen at random from some larger $J\times J$ orthogonal matrix (corresponding to e.g.~the Fourier basis), in which case $\gamma_1=1$ and $\gamma_0\sim \sqrt{J'/J}$.
\end{remark}}

  The assumption given in \eqref{e:Econd} is a condition number type constraint.  The upper bound is simply a bound on the largest eigenvalue of $\bF = \frac{1}{N}\sum_{n} \bE_n^\top \bE_n$, while the lower bound is a bit stronger than a corresponding bound on the smallest eigenvalue.  It is, in essence, saying that the norms $\|\bE_n w\|_2^2$ should be on the order of $\norm{w}_2^2$ for many of the indices $n$. (In contrast, upper and lower eigenvalue conditions would only ensure that $\|\bE_nw\|_2^2$ is on the order of $\norm{w}_2^2$ {\em on average}, but would not prevent degenerate scenarios such as $\|\bE_1 w\|_2^2= N\norm{w}_2^2$ and $\|\bE_2 w\|_2^2=\dots=\|\bE_N w\|_2^2 =0$.)\\


\section{The Dense Setting}\label{s:func}
When the underlying functions are observed at a relatively large number of time points, a different approach than the one described in Section \ref{s:sparse} is commonly employed.  In particular, for each subject $n\in\{1,\dots,N\}$, the observations $Y_{nm}=Y_n(t_{nm})$  are embedded into a function space, which are then treated as though they were fully observed functions.  Implicitly, one is assuming that the error from the embedding is negligible compared to other sources of variability.   More details and background can be found in \citet{hkbook}.   We begin by defining the underlying model.

\begin{assumption} \label{a:f:model}
Let $Y_1,\dots,Y_N$ be independent random elements of a real separable Hilbert space, $\mcH$, satisfying the functional linear model
\[
Y_n  = \sum_{i=1}^I X_{ni} \beta_i + \eg_n.
\]
 Assume the $N \times I$ design matrix $\bX = \{X_{ni}\}$ is deterministic and has standardized columns, the $\eg_n$ are iid~square-integrable Gaussian random elements of $\mcH$ with mean 0 and covariance operator $C$, and $\beta_i$ are deterministic elements of $\mcH$.  Let $\Lambda = (\lambda_1,\lambda_2,\dots)$ denote the vector of eigenvalues of $C$.
\end{assumption}
\edits{\begin{remark}
We emphasize that, throughout this section, the responses $Y_n$, coefficients $\beta_i$, and errors $\eg_n$ are all elements of the Hilbert space (e.g.~functions), rather than scalars as in the usual high-dimensional regression setting. Our work in this section can be viewed as a functional analogue of existing high-dimensional sparse signal recovery results for the scalar setting (i.e.~the ordinary LASSO). 
\end{remark}}

The most common choice for the function space $\mcH$ is $L^2(\mcT)$ where $\mcT$ is a closed and bounded interval.  However, the advantage of our phrasing is that we allow for a number of other spaces, including Sobolev spaces (if one wants to better utilize the smoothness of the data), product spaces (if the response is actually a vector of functions), or multidimensional domains (which could arise in areas such as spatial statistics).  Throughout this section, when we write $\| \cdot\|_{\mcH}$ we will always mean the inner product norm on $\mcH$.  
Norms which deviate from this inner product norms will include alternative subscripts.

As before, when applying the FS-LASSO, we are assuming that the true underlying model is sparse, with $I_0$ denoting the number of true predictors and $S_0$ denoting the true set of predictors.  The FS-LASSO estimate is then the solution to the following minimization problem
\[
\widehat \beta = \argmin_{\beta \in \mcH^{I}} L_{\mathsf{F}}(\beta),
\]
where
\[
L_{\mathsf{F}}(\beta) = \frac{1}{2} \sum_{n=1}^N \| Y_n - X_n^\top \beta\|_{\mcH}^2 + \lambda \|\beta\|_{\ell_1/\mcH} \text{\quad for\quad } \|\beta\|_{\ell_1/\mcH} = \sum_{i=1}^I \|\beta_i\|_{\mcH} \;.\]
The norm $\| \cdot \|_{\ell_1/\mcH}$ is a type of $\ell_1$ norm on the product space $\mcH^{I}$, which encourages sparsity among the list of functions $\beta_1,\dots,\beta_I$ (that is, the function $\beta_i$ will be uniformly zero for many indices $i$).  The target function $L_{\mathsf{F}}(\cdot)$ is a direct Hilbert space generalization of the LASSO target function.  

Next we introduce a functional restricted eigenvalue assumption, which is a direct analogue of similar assumptions in the LASSO literature (e.g. \citet{bickel2009simultaneous}).
%
%
\begin{definition} \label{d:f:comp}
We say a matrix $\bA \in \mbR^{N \times I}$ satisfies a functional restricted eigenvalue condition, $\mathsf{RE_F}(I_0, \alpha)$, if for all subsets $S \subset \{1, \dots, I\}$ with $|S| \leq I_0$, we have
\[
\|\bA x\|_{\mcH^N}^2 \geq \alpha N \| x\|_{\mcH^I}^2
 \ \text{ for all } \  x \in \mcH^I 
 \ \text{ that satisfy }\ \| x_{S^c} \|_{\ell_1/ \mcH} \leq 3  \| x_{S} \|_{\ell_1/ \mcH}.
\]
\end{definition}

In fact, this functional restricted eigenvalue assumption is no stronger than the usual (scalar) restricted eigenvalue assumption---in the following theorem we show that any matrix $\bA$ satisfying the usual (scalar) restricted eigenvalue assumption, will also satisfy the functional version given in Definition \ref{d:f:comp}.  \textcolor{black}{We do this by showing that the common inequality used for proving that a matrix satisfies a scalar restricted eigenvalue condition immediately implies the same inequality in the functional setting, and leads to the restricted eigenvalue condition. } 
\begin{theorem} \label{t:res}
For a fixed matrix $\bX\in\mathbb{R}^{N\times I}$, suppose that for some $c_1,c_2>0$, $\bX$ satisfies
\[
\| \bX z\|_{\mathsf{2}} \geq c_1 \sqrt{N}\cdot \|z\|_2  - c_2 \sqrt{\log(I)}\cdot \|z\|_{\ell_1},
\]
for all $z \in \mbR^I$.  Then the same inequality holds for all $x \in \mcH^{I}$, that is,
\begin{equation}\label{eqn:func_RE}
\| \bX x\|_{\mcH} \geq c_1  \sqrt{N}\cdot \|x\|_{\mcH}  - c_2  \sqrt{\log(I)}\cdot \|x\|_{\ell_1/\mcH} .
\end{equation}{\color{black}
Furthermore, if $c_1>4c_2\sqrt{\frac{I_0\log(I)}{N}}$, then $\bX$ satisfies the $\mathsf{RE}_{\mathsf{F}}(I_0,\alpha)$ property with 
\[\alpha = \left(c_1 - 4c_2\sqrt{\frac{I_0\log(I)}{N}}\right)^2.\]}
\end{theorem}

We are now ready to present the main result of the section.
\begin{theorem}\label{t:highfreq_main}
If Assumption \ref{a:f:model} holds, $\bX$ satifies $\mathsf{RE_F}(I_0, \alpha)$, and
\[
\lambda \geq 2 \sqrt{N}  \sqrt{\| \Lambda\|_1  + 2 \|\Lambda\|_2 \sqrt{\log(I/\delta)}   + 2 \| \Lambda \|_\infty (\log(I/\delta))},
\] 
then with probability at least $1 - \delta$, any minimizer $\widehat \beta$ of $L_F(\beta)$ satifies
\[
\| \bX (\widehat \beta - \beta^\star)\|_{\mcH} \leq \frac{4  \lambda \sqrt{I_0}  }{\sqrt{\alpha N}} 
\qquad \mbox{and} \qquad
\| \widehat \beta - \beta^\star \|_{\ell_1/\mcH} \leq \frac{16  \lambda I_0  }{\alpha N }. 
\]
\end{theorem}
\noindent \textcolor{black}{We mention that since $\Lambda $ is vector of the eigenvalues of $C$ in decreasing order, then one has that
$\|\Lambda\|_2^2 = \sum_{i=1}^\infty \lambda_i^2, \ \|\Lambda_1\|_1 = \sum_{i=1}^\infty \lambda_i, $ and $\|\Lambda\|_{\infty} = \lambda_1$.  The target function $L_F(\beta)$ is convex and coercive, and thus has a least one solution.  As in the scalar case, the function may have multiple solutions, but this still depends heavily on the design $\bX$ and does not happen often in practice.  In the next section we will discuss how this estimate can be computed.}

Interestingly, the rates of convergence for $\mcH$-valued response variables are exactly the same as those for a scalar response.  This is due to our ability to extend scalar concentration inequalities to general Hilbert spaces, see Lemma \ref{l:n_exp_ineq}.   \textcolor{black}{If we take $\lambda \sim \sqrt{N \log(I)}$ then the prediction error becomes $\| \bX (\widehat \beta - \beta^\star)\|_{\mcH} \sim \sqrt{{I_0 \log(I)}}$ while the estimation error becomes $\| \widehat \beta - \beta^\star \|_{\ell_1/\mcH} \sim I_0 \sqrt{ \frac{\log(I)}{N}}$.   Since this result applies to any separable Hilbert space, our methodology is applicable to a wide range of applications beyond just the GWAS we consider here.  From random fields in spatial statistics to brain imaging in fMRI studies, all fall under this umbrella if they are embedded into a Hilbert space.    As a final note, we mention that the $\lambda$ given in Theorem \ref{t:highfreq_main} gives a very tight control of the probability that the two inequalities hold, however, for the convergence rates all that matters is that $\lambda$ has the right order, as is typical in LASSO style results.}

\section{Computational Details} \label{s:comp}
Here we present several computational tools which we utilize and can be found in Matlab code available through the corresponding authors's website.  We begin by providing computational details for the methods in Section \ref{s:func}.  Traditionally, when handling dense functional data, one constructs the functional objects by utilizing basis expansions.  This is the cornerstone of the FDA package in R.  If $e_1(\cdot), e_2(\cdot), \dots$ is a basis of $\mcH$ (e.g. B-splines or Fourier), then we can approximate $L_{\mathsf{F}}$ as
\[L_{\mathsf{F}}(\beta)\approx
\frac{1}{2} \sum_{n=1}^N \sum_{j=1}^J \langle Y_n - X_n^\top \beta, e_j \rangle^2 + \sum_{i=1}^I \sqrt{ \sum_{j=1}^J \langle \beta_i, e_j \rangle^2 },
\]
for $J$ large (often over one hundred).  Letting $B_{i,j} = \langle \beta_i, e_j \rangle$ and $Y_{nj} = \langle Y_n, e_j \rangle,$ we have that the above can be expressed as
\begin{align*}
& \frac{1}{2} \sum_{n=1}^N \sum_{j=1}^J ( Y_{nj} - X_n^\top B_{*j})^2 + \sum_{i=1}^I \|B_{i*}\| \\
& = \frac{1}{2}\| Y - \bA_{\mathsf{F}} \vc(\bB^\top) \| ^2 + \sum_{i=1}^I \|B_{i*}\|,
\end{align*}
where $\bA_\mathsf{F} = \bX \otimes \bI_{J \times J}.$  For large values of $J$, this can quickly become a substantial computational burden.  However, one can use a data driven basis, such as FPCA, so that $J$ can be taken relatively small.  However, we stress that dimension reduction is not our intent, and thus we can choose the number of FPCs to explain nearly all of the variability of the processes.  Using such an approach, it is common to move from 200 B-spline basis functions down to 5-10 functional principal components, 
 with nearly no information loss.   This phrasing now allows us to use the same group LASSO computational tools as in the sparse setting.  However, it is now possible (even in sparse data settings) that $J$ is less than the number observed time points, resulting in $\bA_{\mathsf{F}}$ being smaller than $\bA$ and with a far simpler form.

Since $I$, the number of groups (SNPs), will generally be extremely large, we implement a screening rule that allows us to substantially reduce the potential number of groups (SNPs) we consider when we minimize $L(\cdot)$ or $L_{\mathsf{F}}(\cdot)$ for some fixed $\lambda$.  Using \citet[Theorem 4]{wang2012lasso}, we do the following:
\begin{itemize}
\item (In parallel.) For each SNP $i$, define $A_i\in \mbR^{\sum_n M_n\times J}$ when using the sparse algorithm with entries
\[(A_i)_{nm,j} = X_{ni}e_j(t_{nm})\;.\]
For the dense algorithm define $A_i\in \mbR^{N J \times J}$ as
$
A_i = X^{(i)} \otimes \bI_{J\times J}
$, 
and compute
\[\norm{A_i^\top Y}_2\text{ and }\norm{A_i}_{\mathsf{F}}\;.\]
\item Find any $\lambda_0\geq \max_i \norm{A_i^\top Y}_2$. At penalty parameter $\lambda_0$, the group LASSO solution will be $\widehat{B}_{\lambda_0}=0$.
\item For any $\lambda<\lambda_0$, according to \cite{wang2012lasso},
\[\frac{1}{\lambda_0}\norm{A_i^\top Y}_2 +\left(\frac{1}{\lambda}-\frac{1}{\lambda_0}\right)\norm{A_i}_{\mathsf{F}}\norm{Y}_2<1 \quad \Longrightarrow \quad (\widehat{B}_{\lambda})_{i*}=0\;.\]
\end{itemize}
Therefore, if our aim is to apply a convex optimizer to solve the group LASSO problem with at most $s$ SNPs, then after finding $\lambda_0=\max_i \norm{A_i^\top Y}_2$ we can choose any $\lambda$ sufficiently large so that no more than $s$ SNPs violate the inequality above.  SNPs which do violate the above are then dropped since they will not enter the solution path.  This allows us to make a substantial reduction in the number of predictors.  In our simulations and application we could handle on the order of 10000 SNPs jointly when fitting the FS-LASSO using the ADMM procedure, with the sparse tools requiring around 30gb of RAM, and the smooth requiring around 4gb.  By using this screening rule, we can reduce millions of predictors to tens of thousands, and then do a final fit using any number of convex optimization routines.  

Finally, we mention the choice of the smoothing parameter $\lambda$.  In our simulations we utilize the BIC, though we include several other options in the application section.  We calculate the BIC as
\[
\log(\hat \sigma^2) \sum M_n + J I_{active} \log(N) \qquad \text{or} \qquad
\log(\hat \sigma^2) N M_{pc} + M_{pc} I_{active}  \log(N),
\]
for the sparse and dense methods respectively.  Here $I_{active}$ is the number of predictors in the current model.  The error $\hat \sigma^2$ is calculated as 
\[
\frac{1}{\sum M_n} \sum_n \sum_m (Y_n(t_{nm}) - \hat Y_n(t_{nm}))^2
\qquad \text{or} \qquad
\frac{1}{N M_{pc}} \sum_n \sum_j (Y_{nj} - \hat Y_{nj})^2,
\]
for the sparse and dense methods respectively.  The predicted values are computed by recomputing the corresponding least squares estimates, to eliminate the effect of the bias.  Given the dependence in the data, the BIC is an adequate, though not optimal choice.  Further work is needed on tuning parameter selection for functional models, but we leave this for future research.  Fitting the model for a grid of values for $\lambda$ can be done relatively quickly by utilizing a warm start, i.e. using the previous solution for a particular $\lambda$ value as the starting point for finding the next solution.

\section{Simulations} \label{s:sims}
In this section we present a simulation study to compare the performance of the discussed methods.  The predictors, $X_{ni}$, are generated from a normally distribution with $\Cov(X_{ni},X_{n'j}) = 1_{n = n'} \rho^{|i - j|}$, which is an autoregressive covariance with the rate of decay controlled by $\rho$ (with independence across subjects, i.e.\ $X_{n*}$ and $X_{n'*}$ are independent for $n\neq n'$).  
\textcolor{black}{We mention that in our genetic application the predictors take values 0/1/2 and are thus not normal.  Additional simulations using Binomial predictors can be found in Appendix \ref{a:binom}, and while the performance of all of the methods presented here decrease, the relative conclusions from comparing the methods stay the same, we thus focus on the normal case here.}
  We take $I=1000$ and $I_0=10$.  
The nonzero functions $\{\beta_i(t): t\in [0,1]\}$ are randomly generated from the Mat\'ern process with parameters $(0,I_0^{-1},1/4,0,5/2)$ and the errors $\{\eg_{n}(t) :t\in [0,1]\}$ are generated in the same way but with parameters $(0,1,1/4,0,3/2)$, which results in errors that are less smooth than the parameter functions.  \textcolor{black}{The Mat\'ern process is very flexible family of stationary processes which produce more realistic structures for biological applications, as compared to Brownian motion or simple low dimensional structures.}  Each subject is observed at 10 uniformly distributed locations which differ by curve.  We consider $N =  50, 100, 200$ and $\rho =0.5$ and $0.75$, and use 1000 repetitions of each scenario.  For the functional method, FPCA is carried out using the PACE package in Matlab.  For the sparse method we use cubic Bsplines with $J = 30$.

To compare the methods without having to worry about tuning parameter selection, we examine smoothed ROC curves which give the proportion of true positives found as a function of the false positives.  The curves for all scenarios are given in Figure \ref{f:roc}. \textcolor{black}{ Surprisingly, the ROC curves for the two method are about the same, meaning that in terms of variable selection the sparse and functional methods have nearly equivalent performance on average, though, as we will see in the application section, for a single iteration they can still disagree on the selected subset.}  \textcolor{black}{Next we examine the average prediction error in Figure \ref{f:error_time}, with the penalty parameters chosen by BIC, AIC, and 2-fold cross validation.  There we see that the sparse method has an advantage, 
though this decreases for large sample sizes.  The BIC and AIC perform about the same, while the CV criteria seems to perform the worst.  We note that, since the prediction error compares linear combinations of the parameter estimates, we use it as a single number proxy summarizing the estimation error as well, though clearly if there were particular patterns for $\beta(t)$ one was interested in examining (linear, sinusoidal, etc.), it would be interesting to include them in the simulations and examine their estimation error, however we don't purse this further here.}

\begin{figure}[h]
\centering   
\includegraphics[width=0.49\textwidth]{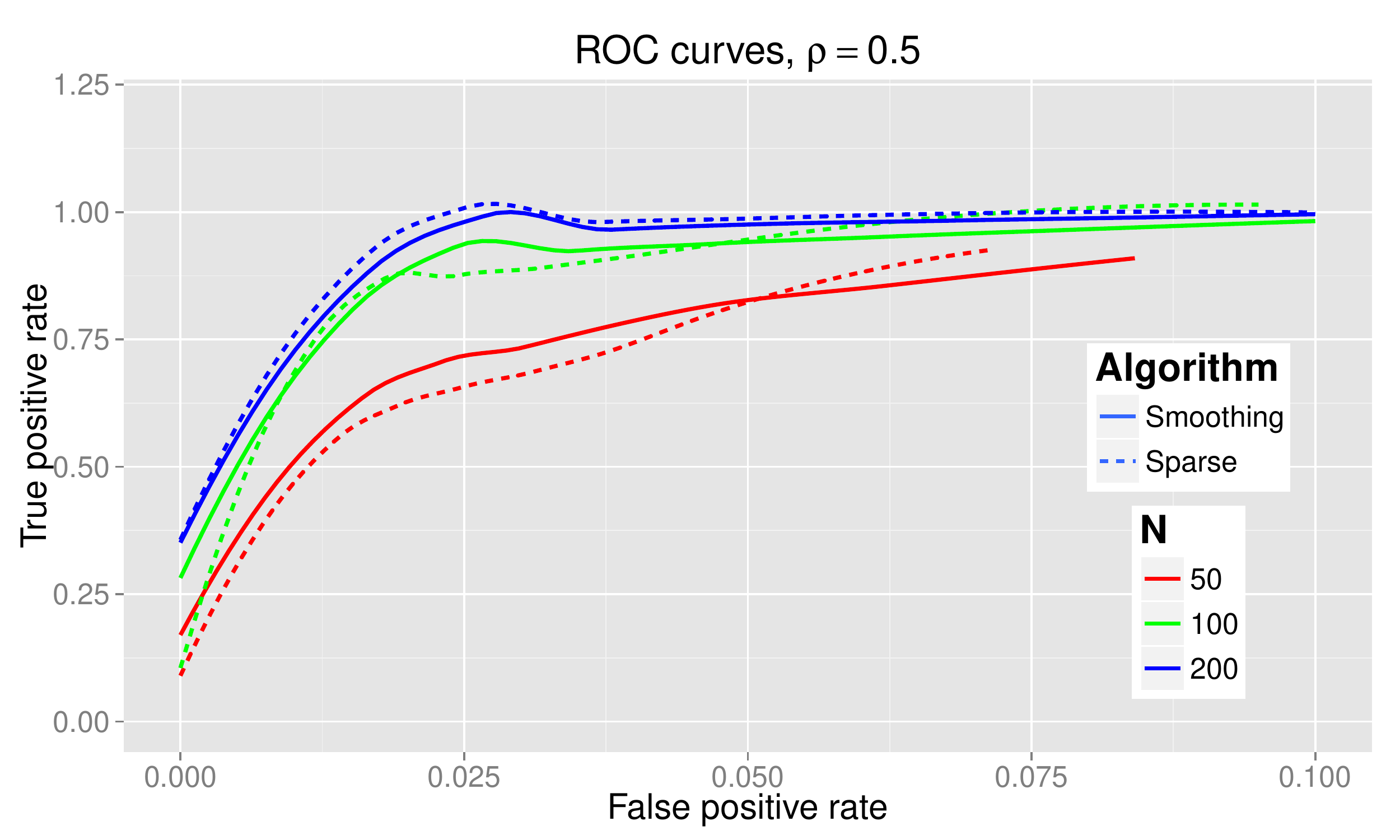}
\includegraphics[width=0.49\textwidth]{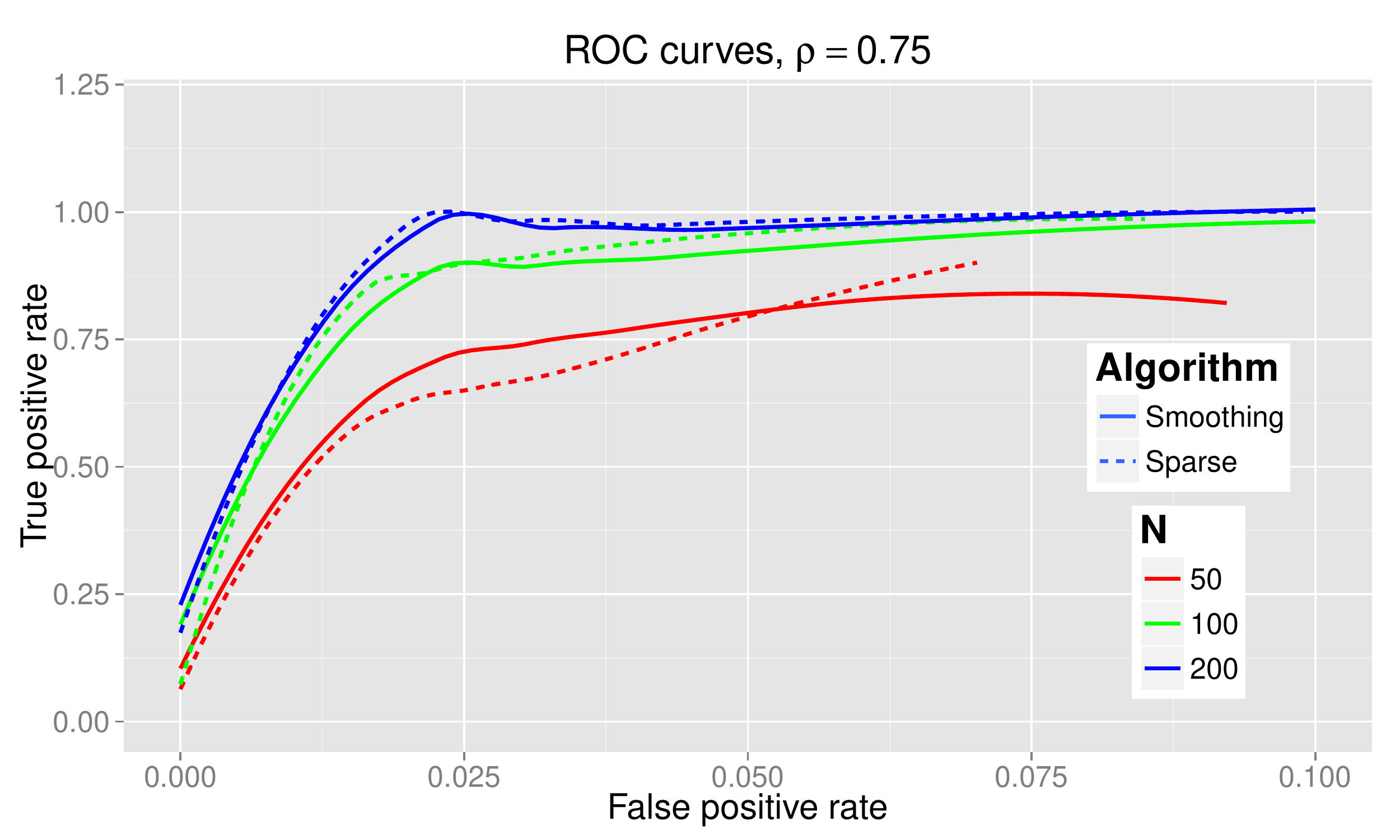}
\caption{ROC curves comparing the smooth and sparse methods.   \label{f:roc}}
\end{figure}

\begin{figure}[h]
\centering  
\includegraphics[width=0.49\textwidth]{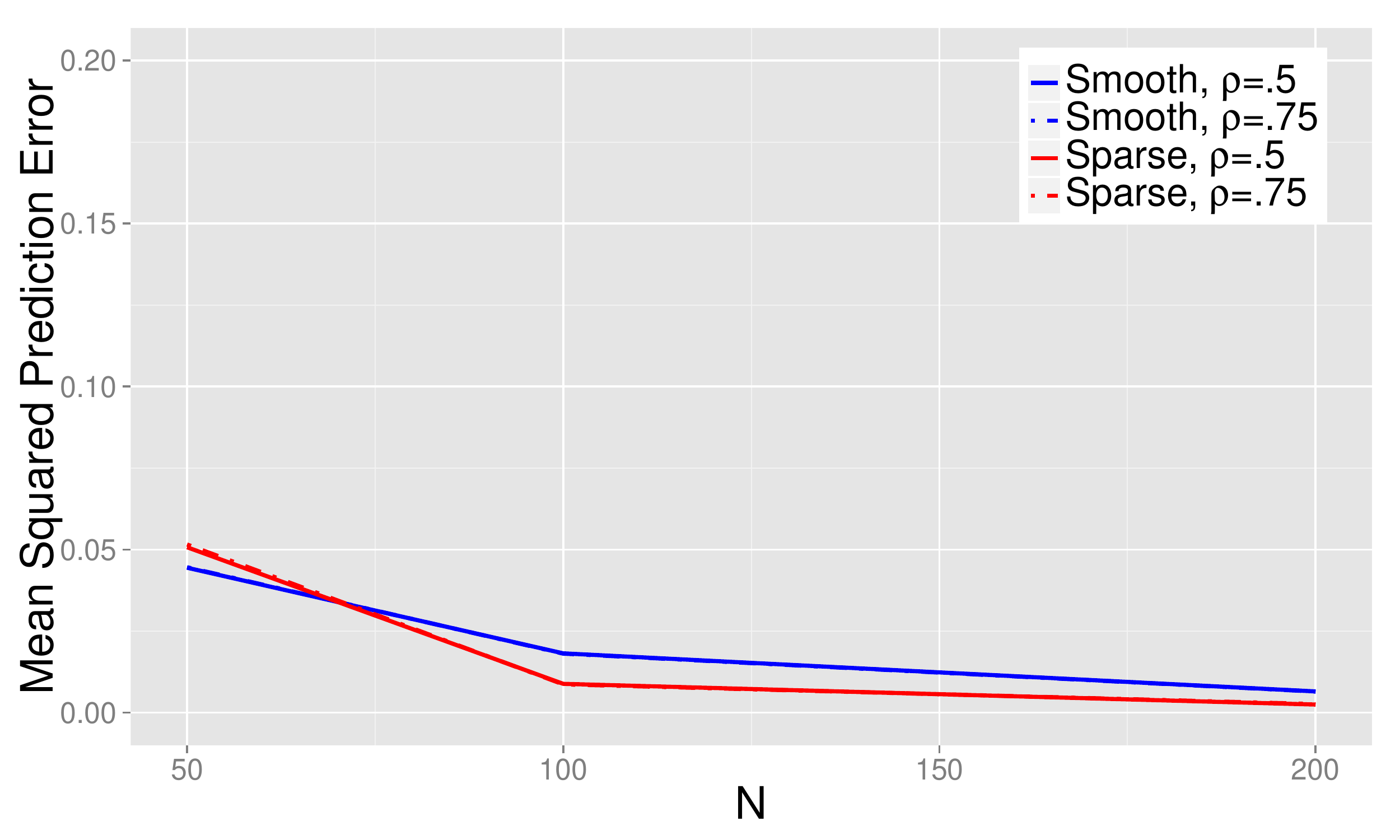} 
\includegraphics[width=0.49\textwidth]{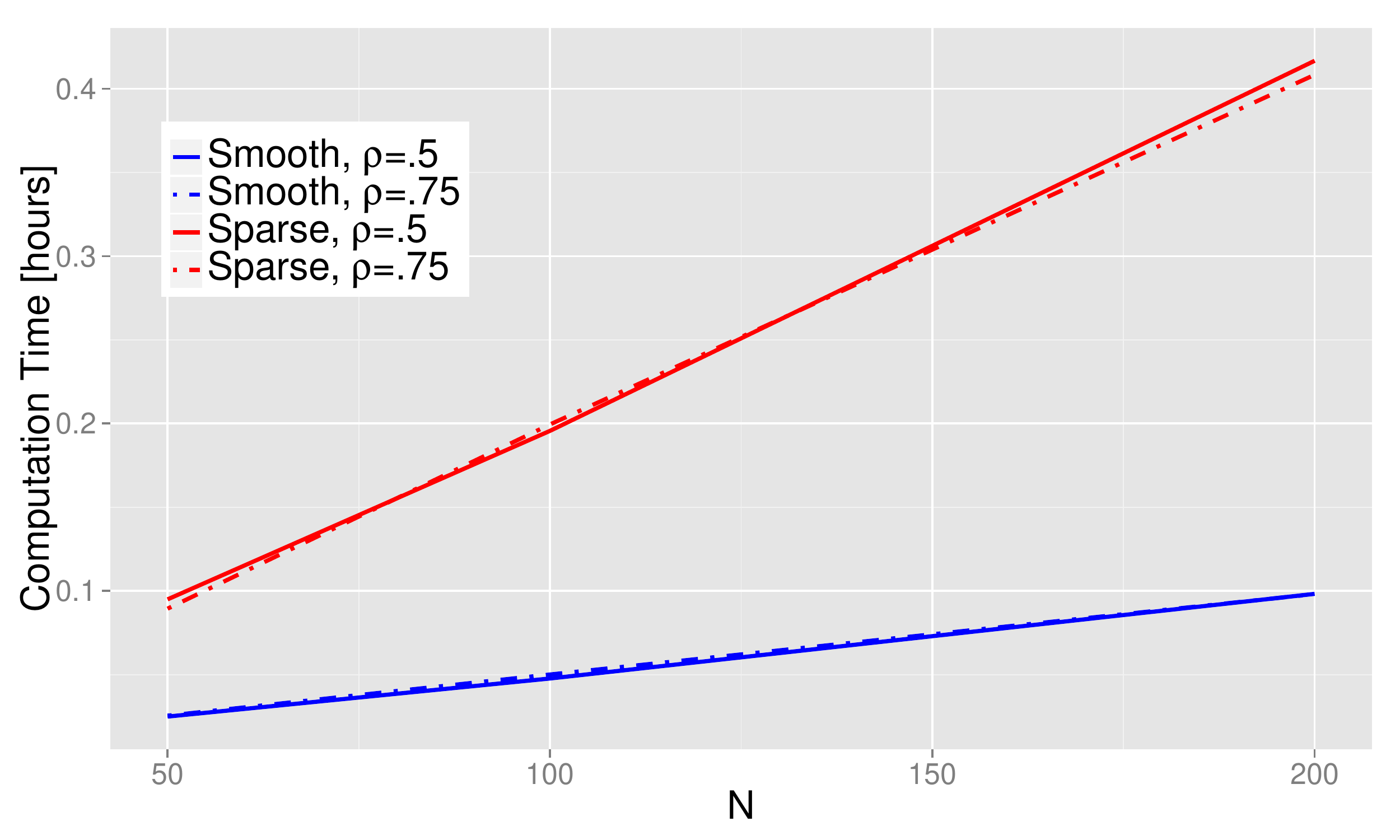} 
\includegraphics[width=0.49\textwidth]{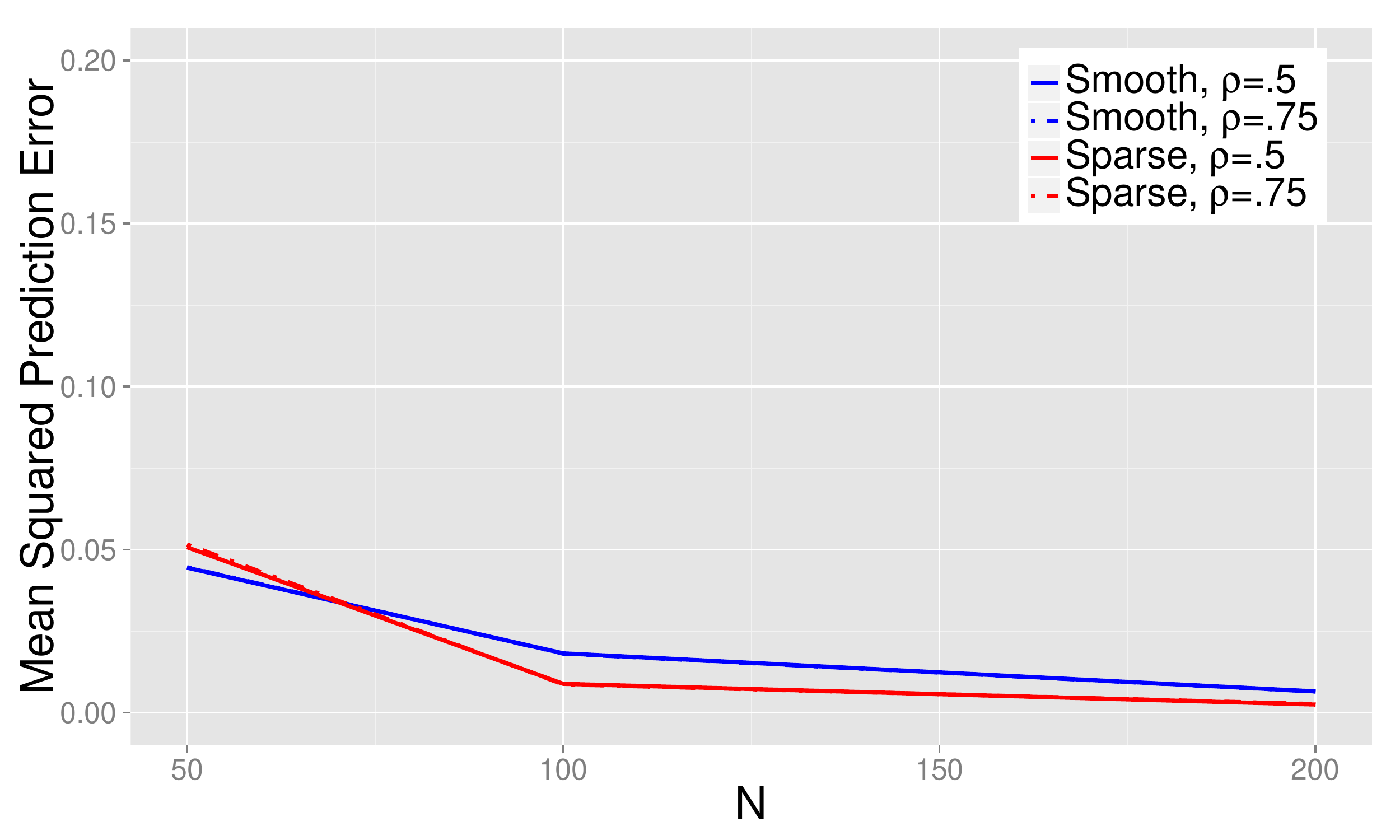} 
\includegraphics[width=0.49\textwidth]{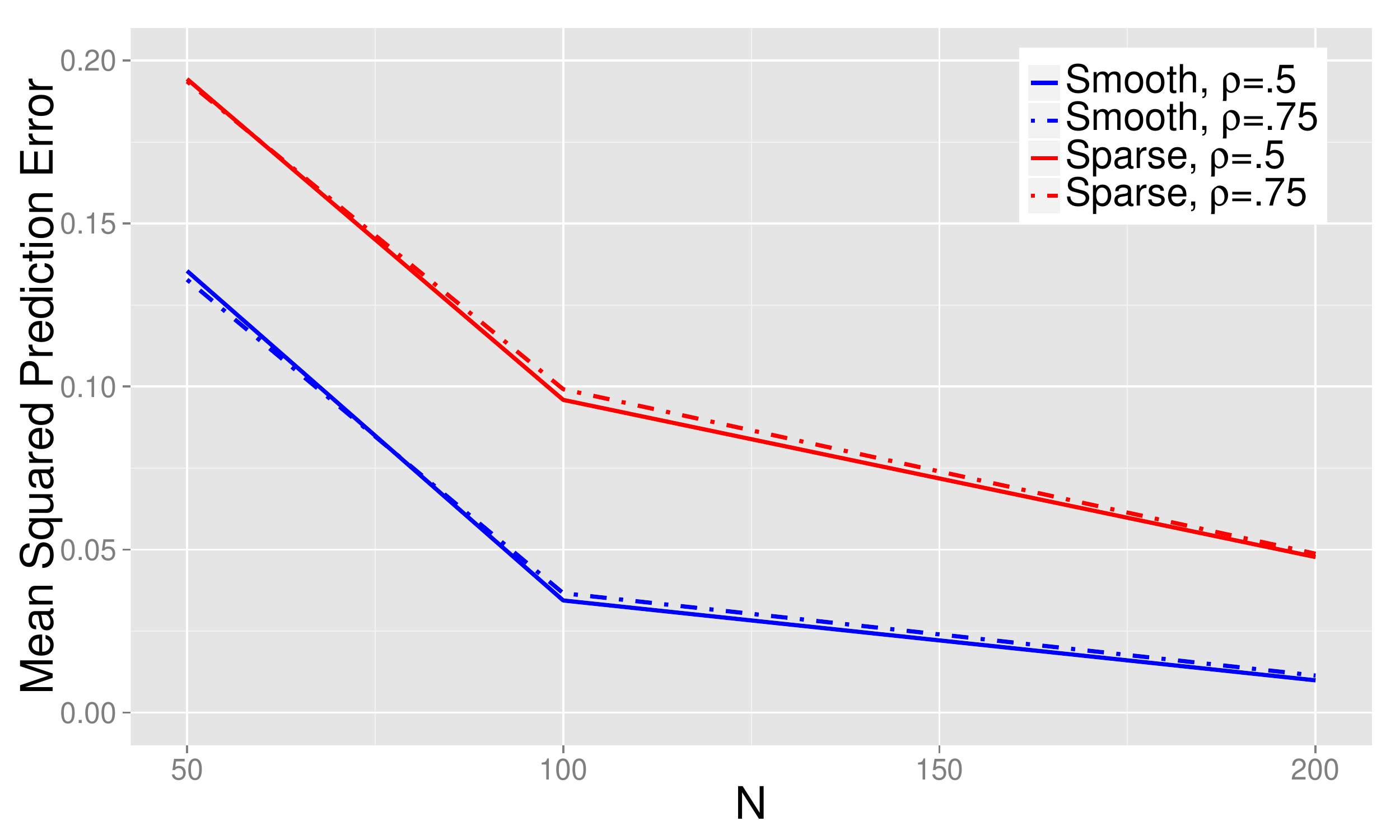} 
\caption{Prediction error and computation times for smooth and sparse methods.   The top row is the MSE and computation time for the BIC, while the first panel on the second row is the MSE for AIC, and the second panel on the second row is the MSE for cross-validation.  \label{f:error_time}}
\end{figure}

\textcolor{black}{
Next we compare to two forms of scalar LASSO, one which ``time-corrects" by incorporating a nonparametric estimate of the mean (TC-LASSO) and one which does not.  Using the nonparametric smoother is an attempt to account for the non-stationarity of the processes.  The resulting ROC curves are given in Figure \ref{f:reg} for $\rho = 0.5$ and $\rho = 0.75$, with the solid lines indicating scalar LASSO and the dashed lines the time-corrected version.  Examining the plots we see a potentially interesting pattern.  For smaller false positive rates, LASSO is actually doing better than FS-LASSO (Figure \ref{f:roc}).  However, as one moves along the x-axis (the false positive rate), the neither version of LASSO changes much, while FS-LASSO quickly tops out, finding all significant predictors.  This suggests that LASSO is able to find a subset of predictors rather quickly, likely those which do not vary much over time, but is unable to capture the more complex signals.  Surprisingly, incorporating a nonparametric estimate of the mean (TC LASSO) actually results in a large decrease in power.  This further supports the need for FS-LASSO, if one believes there is nonlinear structure in the data, then patching LASSO is substantially worse than moving to a functional framework.}

\begin{figure}[h]
\centering  
\includegraphics[width=0.49\textwidth]{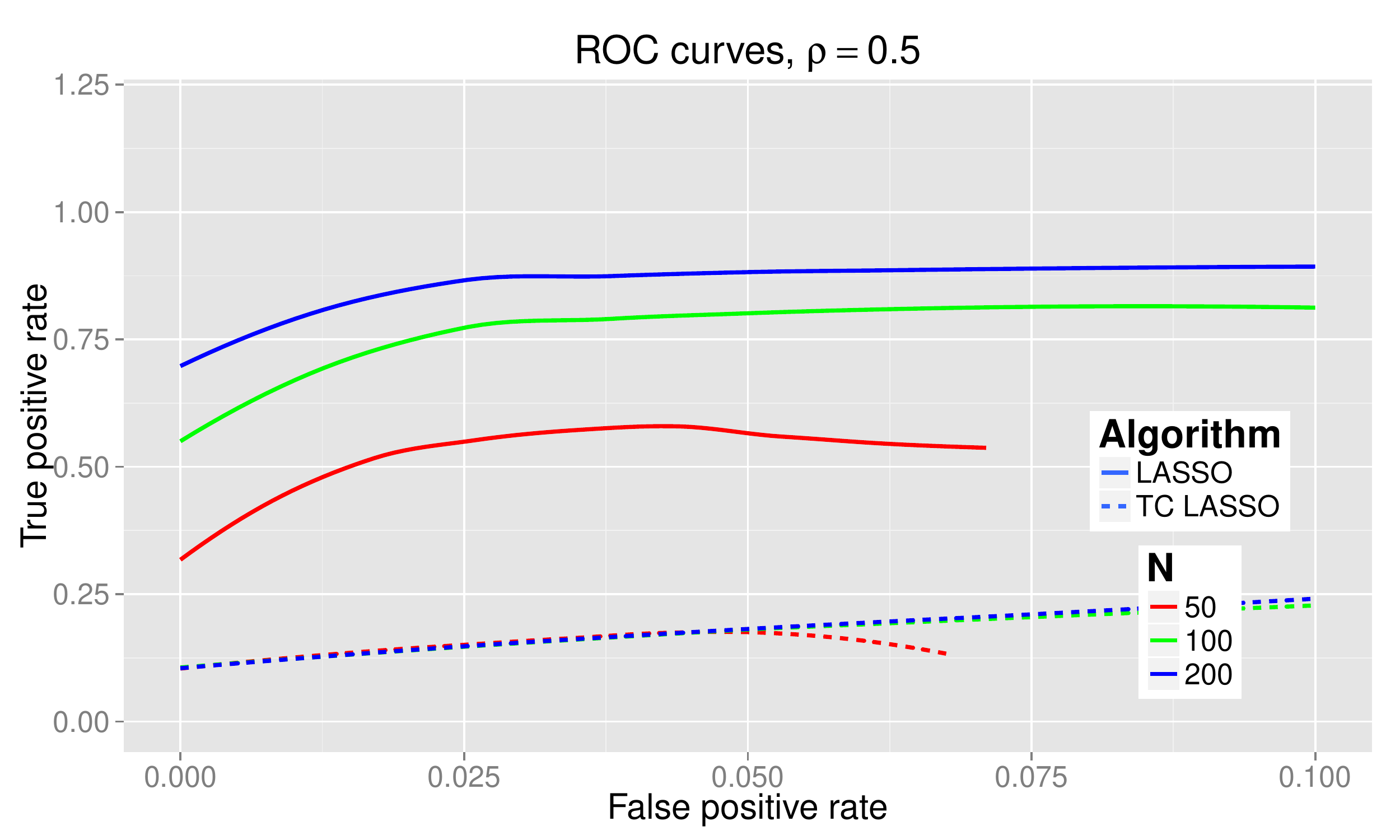} 
\includegraphics[width=0.49\textwidth]{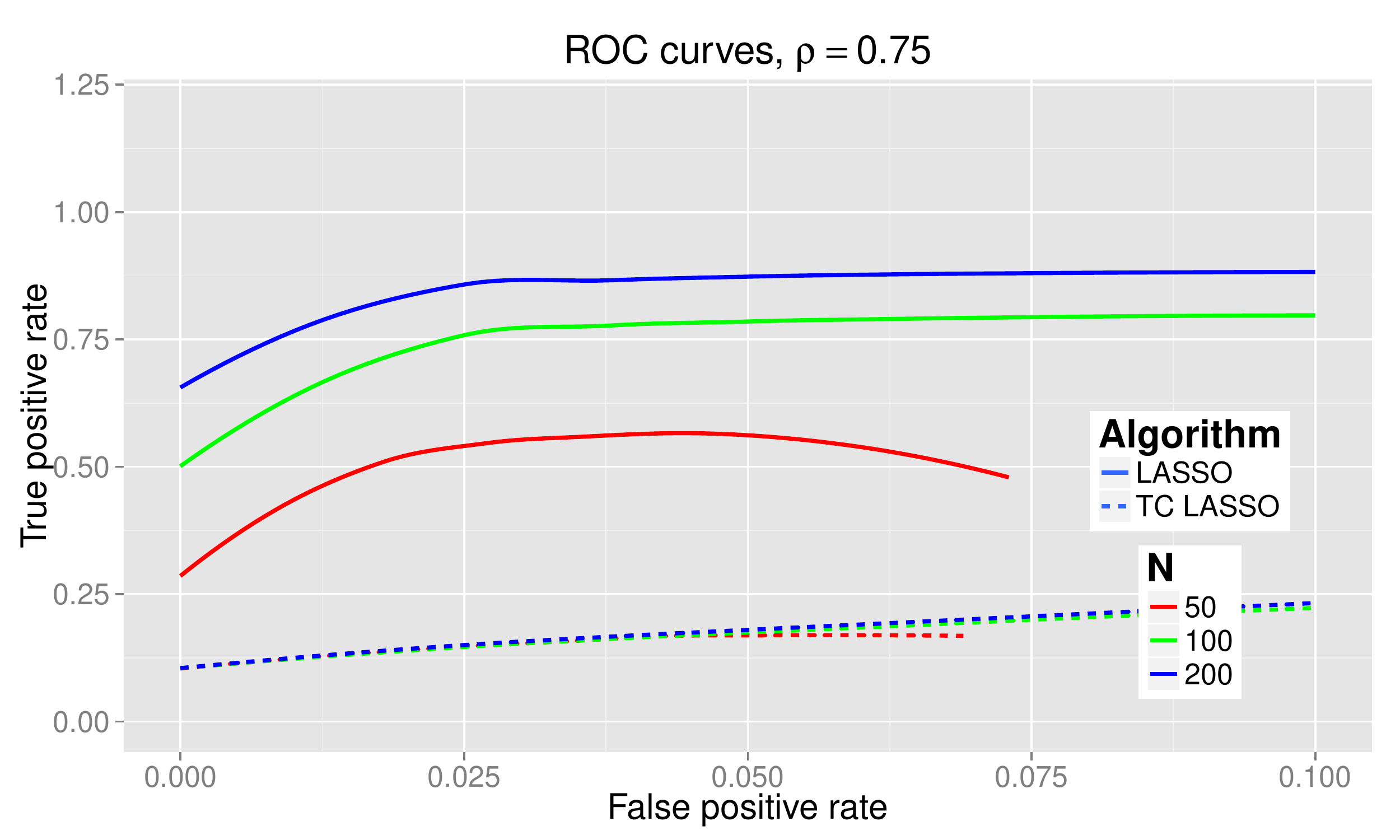} 
\caption{ROC curves for scalar LASSO with (TC LASSO) and without (LASSO) a nonparametric estimate of the mean.  \label{f:reg}}
\end{figure}

\textcolor{black}{As a final comparison of statistical performance, we examine the sensitivity of the sparse FS-LASSO on the choice of $J$.  We consider $J=  27$ and $J = 33$, which constitutes a $10\%$ change in the number of basis functions.  The resulting MSE curves are given in Figure \ref{f:basis}.  Unfortunately, there does appear to be some sensitivity in the performance of FS-LASSO relative to the number of basis functions.  While $J = 33$ and $J = 27$ perform similarly, both are slightly worse than $J=  30$, and in fact, worse than the smoothing algorithm.  This can be remedied by also choosing the number of basis functions by BIC, AIC, or CV, or by using the smoothing method to select the predictors.  Namely, one would ideally like to use the penalty to also control the level of smoothing (not just the selection of the predictors).  FS-LASSO behaves very similarly to a ridge regression on the selected predictors, but a penalty which acts more like a smoother would likely perform better (at least in terms of MSE).  However, we do not explore this further here as this is an important topic of further research.}

\begin{figure}[h]
\centering  
\includegraphics[width=0.49\textwidth]{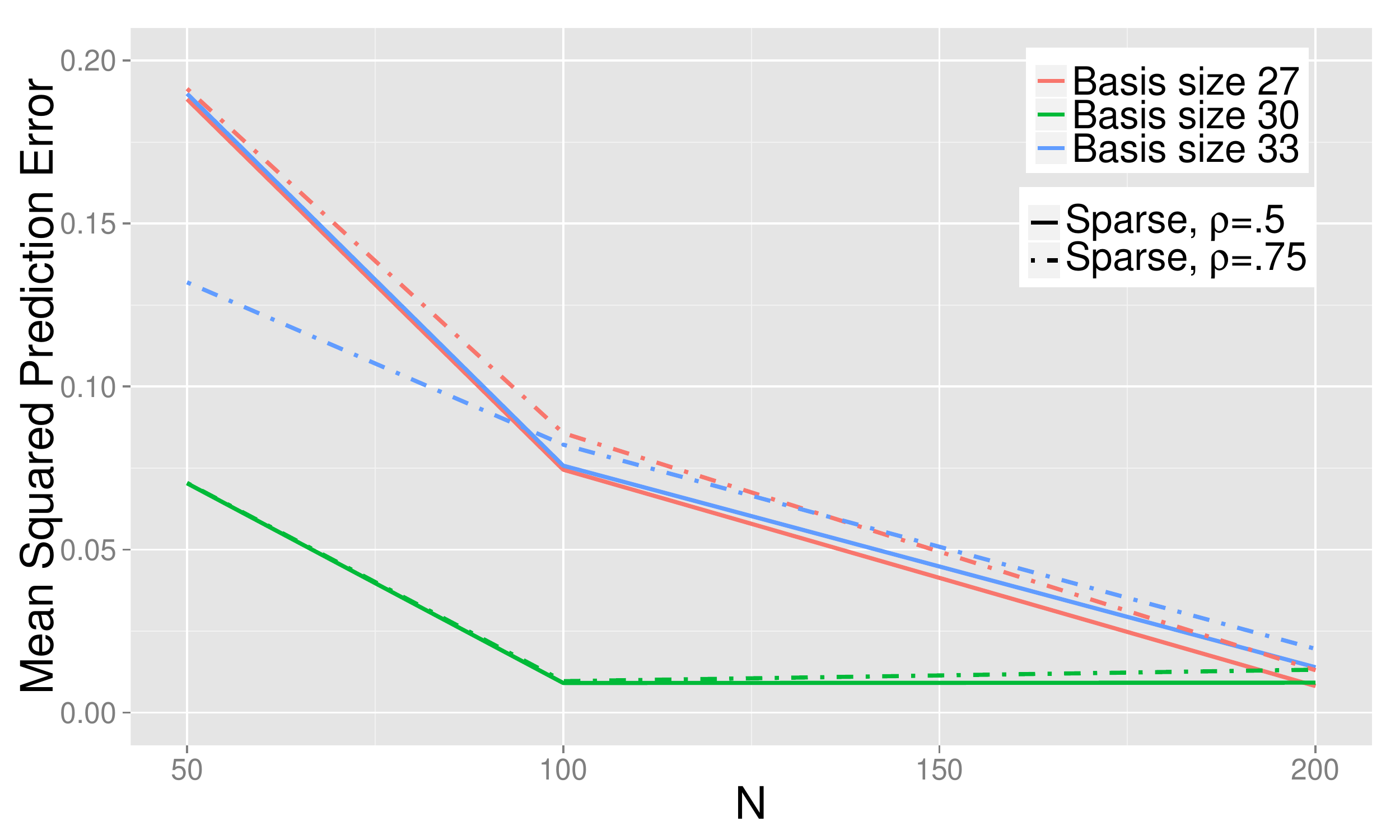} 
\caption{Mean squared prediction error for sparse FS-LASSO with differing numbers of basis functions.  \label{f:basis}}
\end{figure}

Lastly, we mention computation time and memory.  Methods which scale well are critical for applications such as genetic association studies.  There, one needs to work with millions of predictors and both computation time and memory become major issues.  On a desktop computer, the sparse method presented here can be applied with predictors of the order of 10,000 or so, but not much higher.  The smoothing method, which allows for the inclusion of an FPCA, can be applied with predictors of the order of 20,000.  On clusters these numbers can be increased depending on the available memory.  In Figure \ref{f:error_time} we plot average computation times for the two methods when taking $I = 10000$, but keeping everything else the same.  We see that the smooth method has a substantial edge in computation time, resulting in shorter times which also scale better with sample size.

\section{Framingham Heart Study} \label{s:FHS}
\begin{figure}[h]
\centering  
\includegraphics[width=0.48\textwidth]{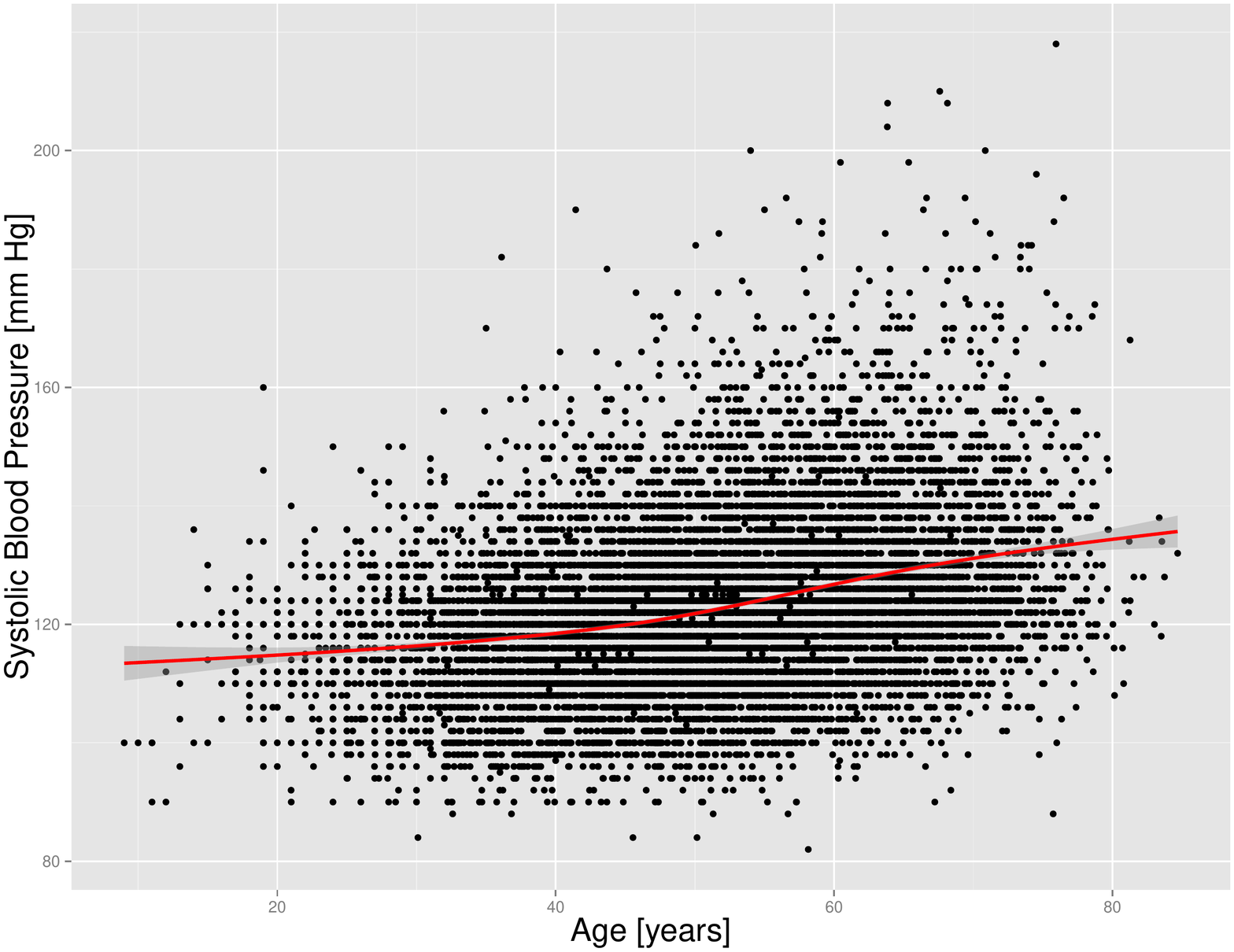}
\includegraphics[width=0.48\textwidth]{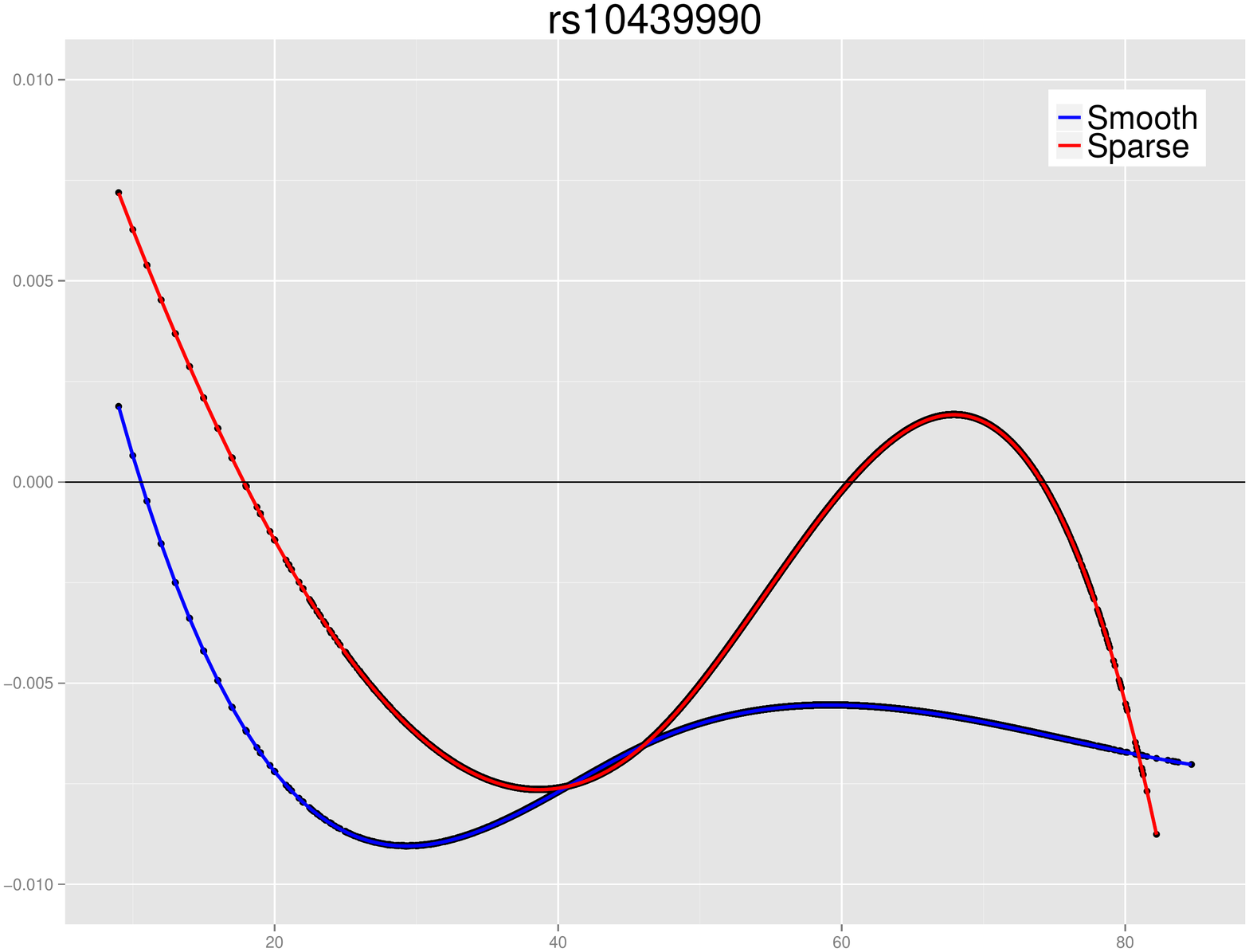} 
\caption{Left Panel: Systolic blood pressure as a function of age from the FHS.  The red line is a local linear smoother estimating the mean. Right Panel: Plot of estimated effect for SNP rs10439990 using the `smooth' algorithm (blue) and the `sparse' algorithm (red). \label{f:sbp}}
\end{figure}
The Framingham Heart Study, FHS, is a long term longitudinal study with the goal of better understanding the risk factors for heart disease.  Genome--wide SNP data and phenotype information were downloaded from dbGaP  (http://www.ncbi.nlm.nih.gov/gap) study accession phs000007.v25.p9.  The study consists of three cohorts of subjects with the first cohort recruited in 1958.  FHS has had a tremendous impact on our understanding of health and risk factors for heard disease.  Since a full account would take far too much space, we refer the interested reader to \cite{o2008cardiovascular,mahmood:2014,chen:2016} for details on different findings and impacts due to FHS.  
Here we examine the second cohort which consists of 1924 subjects.  Each subject contributed up to 7 clinical exams over 29 years, though some subjects passed away during the study and contributed less.  The subjects were genotyped using the Illumina Omni 5M platform resulting in approximately 4.3 million SNPs genotyped.  Our goal here is to find a subset of SNPs impacting systolic blood pressure, a primary risk factor for heart disease.  As an illustration, in Figure \ref{f:sbp} we plot all of the blood pressure measurements versus age and include a local linear smoother.  We see that (as is well known) blood pressure increase monotonically, though nonlinearly, with age.  \textcolor{black}{When we combine this fact with the high within-subject correlation of blood pressure, we see that scalar procedures such as LASSO are inappropriate.}  When applying our procedures, we first remove the effects of gender, height, and HDL cholesterol levels nonparametrically using a local linear smoother.

Due to the size of the data, we cannot apply the methods to the entire set jointly.  We thus first use our method to screen the number of predictors down to a computationally manageable size.  One could take the top performers from marginal regression, but instead we use our procedure to rank the SNPs according to when they enter the model (as one varies the tuning parameter), and take the top ones.  First, we divide up the genome into segments of 10K SNPs for the sparse methods and 20K for the smooth.  Using our algorithm, we rank the SNPs in each segment according to the order in which they enter the model (as one decreases the tuning parameter).  
 We take the top ranking SNPS so that, when pooled, we have screened down to approximately 10K SNPs for the sparse and 20K for the smooth.  A final application of our methods are applied to these SNPs with the smoothing parameter chosen in four different ways: AIC, BIC, extended BIC with parameter $0.2$, and 2-fold cross-validation.  The top SNPs for each method are presented in Table \ref{t:SNPs_sparse_and_smooth}.  As we can see, the smoothing method has found substantially more significant SNPs than the sparse.  However, there were 4 candidates which appeared in both lists.  \textcolor{black}{Given how similar the sparse functional methods compared in simulations, it might at first be surprising that they only overlap on four SNPs, however, given the huge number of predictors, it seems reasonable that there would be substantial differences between the two.  As an illustration, in Figure \ref{f:sbp}, we provide a plot of one of the estimated SNPs, rs10439990, which was selected using both algorithms.  There we see that the effect is negative, meaning that a subject with one or two copies of the mutation has a lower chance of having high blood pressure.  However, the effect is not constant over time as it starts lower a younger ages, 20-40, and rises as subjects age, peaking around 70 or so.  While the two methods produce similar plots, they disagree on the magnitude of the effect later in life, with the sparse algorithm indicating that the protection is essentially gone, while the smooth algorithm indicates the protection is still present, though diminished.  However, we stress that the two submodels selected were different, and thus the difference in estimates could be due in part to having different predictors.  This plot also illustrates a cautionary note about interpreting endpoints.  With so few observations at the ends, it is safer to focus interpretations on sections with the most observations.  }
 
\textcolor{black}{ 
Quantifying the uncertainty of the output from these types of high-dimensional procedures is an ongoing area of research even in the scalar setting.  However, to test if either method is dramatically over fitting, we also applied a 10-fold cross-validation to compare the predictive performance between the sparse and smooth method.  The resulting values were $0.00961$ for the sparse approach and $0.00972$ for the smooth approach, thus showing that the sparse approach is doing a slightly better job in terms of prediction.  Given the inherently sparse nature of the data, it is maybe surprising that the smooth approach was still so close in terms of prediction.  Regardless, as is now common in genetic studies, a follow up on a separate data set would be required to validate the results.  Note that for serious follow up studies, it would be useful to extend techniques such as \cite{shi:2015} to the functional setting, which attempt to recover some of the relatively small effects that LASSO style procedures shrink to zero.  This is especially useful in genetic studies as the effect sizes are relatively small.  Lastly, an extension incorporating dependence between curves would useful as the FHS data included some related individuals, though we do not pursue this further here.  
}

As a final form of validation, we examined the different association results found in the literature for our selected SNPs.  This was accomplished using GWAS Central, \url{http://www.gwascentral.org/index}, which provides a searchable database of GWAS results.  There were at least a few results that validate some of our findings.  In particular, SNP rs10439990 is located the gene ZBTB20, which has been associated with Triglyceride levels \citep{kathiresan:2007}, a common risk factor for heart disease.  SNP kgp29965466 is located on the gene ATRNL1, which is associated with several negative health outcomes and indicators including nicotine dependence \citep{bierut:2007}.  SNP rs10497371 is on gene MYO3B, which has been associated with Diabetes \citep{scott:2007}.  Lastly SNP rs7692492 is on gene FAM190A, which has been associated with coronary heart disease \citep{samani:2007,larson:2007}.

\newcolumntype{M}{@{}>{\columncolor{white}[0pt][0pt]}c@{}}

\begin{figure}[ht]
\centering
\begin{tabular}{l@{\,}|M|M|M|M||l@{\,}|M|M|M|M}
\hline
\multicolumn{5}{c}{Top SNPs selected by sparse algorithm}&\multicolumn{5}{c}{Top SNPs selected by smooth algorithm}\\
  \hline
Chromosome, Name & \,AIC\,& \,BIC\, & \,EBIC & \, CV\, & Chromosome, Name & \,AIC\,& \,BIC\, & \,EBIC & \,CV\, \\ 
  \hline
2, rs3845756 & \cellcolor{gray}  &   & &   & 1, kgp5933154 & \cellcolor{gray}  & \cellcolor{gray}  &  &   \\ 
   \hline
   2, rs6414023& \cellcolor{gray}  & \cellcolor{gray}  &  &   & 1, rs17107710 & \cellcolor{gray}  & \cellcolor{gray}   & &   \\ 

   \hline
   2, rs10497371 &   &   &   &  & 2, kgp5982336 & \cellcolor{gray}  & \cellcolor{gray}  &  &  \\ 

   \hline
3, \textcolor{black}{rs10439990 }& \cellcolor{gray}  & \cellcolor{gray}  & &    &3, kgp26868438 & \cellcolor{gray}  & \cellcolor{gray}  & \cellcolor{gray} & \cellcolor{gray} \\ 

   \hline
4, rs7692492 &   &   &  & & 3, \textcolor{black}{rs10439990}& \cellcolor{gray}  & \cellcolor{gray}  & \cellcolor{gray} &       \\ 

   \hline
5, \textcolor{black}{kgp30202888} & \cellcolor{gray}  &   &  & & 3, kgp11928513 & \cellcolor{gray}  & \cellcolor{gray}  &     \\ 
 
   \hline
7,   \textcolor{black}{kgp3808198} & \cellcolor{gray}  &   & & & 5,  \textcolor{black}{kgp30202888} & \cellcolor{gray}  & \cellcolor{gray}  & \cellcolor{gray} &    \\

   \hline
7,  \textcolor{black}{kgp4510449} &   &   &   &  & 7,  \textcolor{black}{kgp3808198} & \cellcolor{gray}  & \cellcolor{gray}  &    \\  

   \hline
8, kgp8137960 &   &   &   & &   7,  \textcolor{black}{kgp4510449} & \cellcolor{gray}  & \cellcolor{gray}  & \cellcolor{gray} &    \\ 

   \hline
9, rs1702645 &   &   &   & & 10, kgp29965466 & \cellcolor{gray}  & \cellcolor{gray}  & \cellcolor{gray} & \cellcolor{gray}   \\ 

   \hline
12, rs978561&   &   &   & &  11, kgp10123049 & \cellcolor{gray}  & \cellcolor{gray}  &     \\ 

   \hline
13, rs1924783 & \cellcolor{gray}  & \cellcolor{gray}  &  &  &12, kgp6953877 & \cellcolor{gray}  & \cellcolor{gray} & \cellcolor{gray}  &     \\ 

   \hline
15, kgp6228266 &   &   &   & & 12, rs10859106  & \cellcolor{gray}  & \cellcolor{gray}  & \cellcolor{gray} &    \\ 

   \hline


   \hline
\end{tabular}

\caption{Top SNPs selected by the Sparse Algorithm and by the Smooth Algorithm.  Gray boxes indicated the SNP is selected when using the corresponding variable selection criteria.  Red indicates SNPs which were chosen by both the smooth and sparse methods. \label{t:SNPs_sparse_and_smooth}}
\end{figure}

\section{Conclusion}\label{s:conc} 
We have provided powerful new tools for analyzing functional or longitudinal data with a large number of predictors.  In the sparse case we provided new theory in the form of a restricted eigenvalue condition and accompanying asymptotic theory.  While phrased as a functional data method, the sparse case is closely related to varying coefficient models, and thus this work can be viewed as an extension of the work in \cite{wei:huang:li:2011}.  In the dense case, we provide a completely new methodology and accompanying asymptotic theory which allows the response functions to take values from any separable Hilbert space.  Such generality means the methods can be applied to a variety of settings including traditional functional outcomes, functional panels, spatial processes, and image data such as fMRI.  We also provide accompanying computer code which takes advantage of the structure of group LASSO so that our methods can be applied efficiently.

Our simulations suggest that the choice between sparse and smooth tool sets is not straightforward.  The simulations were done in a traditionally sparse setting, but the two methods were nearly equivalent in terms of variable selection and the functional method required far fewer computational resources.  However, estimation accuracy was better for the sparse methods.  This opens the door to future work which could involve multistage methods where selection is done via smooth methods and estimation via sparse methods.  

While our methods perform well, they also highlight the need for better tuning parameter selection methods.  
Methods such as BIC and cross-validation can be applied in this setting with good results, but such methods are not tailored to the dependence inherent in functional data, and we hope to investigate possible adaptations to these selection criteria for the functional setting in future work.


\bibliographystyle{abbrvnat}
\small
\bibliography{flasso}
\normalsize

\clearpage
\def\@shorttitle{Testing}
\appendix

\section{Binomial Simulations}\label{a:binom}
\textcolor{black}{
Here we carry out a nearly identical simulation study as in Section \ref{s:sims} except the normally distributed predictors are replaced with $B_{in} \sim Binom(2,p_i)$ random variables to more closely emulate the structure of the application.  To generate $B_{in}$ which are correlated, we use a probit model based on the the same simulation structure as in Section \ref{s:sims}; we take those normal random variables as the input in probit model where $p_i = \logit^{-1}(X_{in})$.  The MSE curves are summarized in Figure \ref{f:mse_b}.  If we compare these results with those from the bottom right panel from Figure \ref{f:error_time}, here we also use CV, we see that the pattern is very similar.  The error does not change dramatically between the two correlations, but the sparse method has a noticeably lower MSE.  
}

\begin{figure}[h]
\centering  
\includegraphics[width=.49\textwidth]{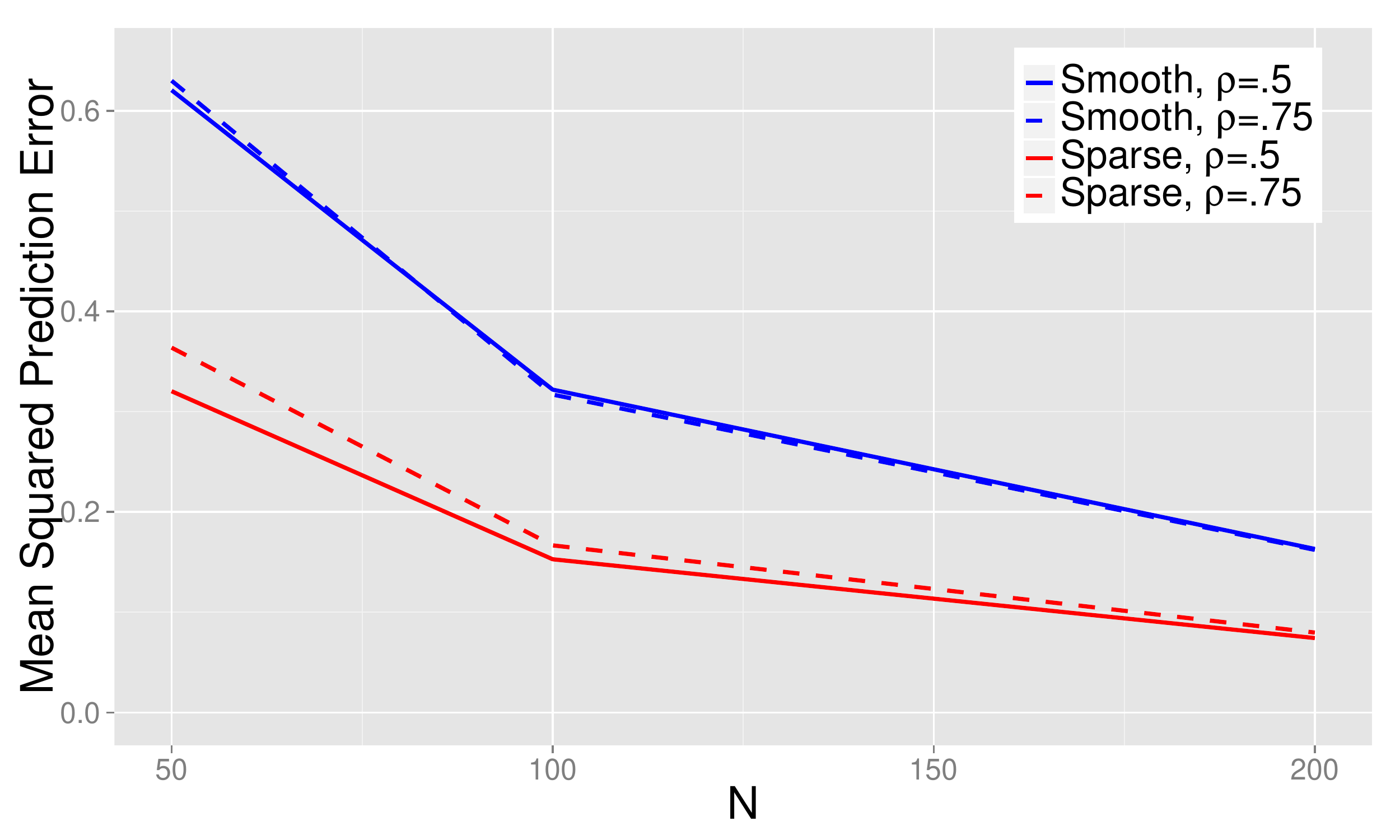}
\caption{MSE curves comparing the smooth and sparse methods with $Binom(2,1/2)$ predictors.  Tuning parameter is selected with cross-validation. \label{f:mse_b}}
\end{figure}
\section{Proofs for Section \ref{s:sparse}}

\subsection{Proof of Theorem~\ref{thm:main_sparse}}\label{a:proof_sparse}
Our analysis for this Theorem closely follows standard techniques for sparse and group sparse regression
under restriced eigenvalue conditions, such as the work by \citet{lounici2011oracle} and \citet{bickel2009simultaneous}.

Recall that $\widehat{\bB}$ is a minimizer of
\[L(\bB) = \frac{1}{2} \| Y - \bA \vc(\bB^\top)\|_2^2 + \lambda \| \bB\|_{\ell_1/\ell_2}\;.\]
Then in particular, we must have $L(\widehat{\bB})\leq L(\bB^\star)$, that is,
\[ \frac{1}{2} \| Y - \bA \vc(\widehat\bB^\top)\|_2^2 + \lambda \| \widehat\bB\|_{\ell_1/\ell_2}
 \leq \frac{1}{2} \| Y - \bA \vc(\bB^\star{}^\top)\|_2^2 + \lambda \| \bB^\star\|_{\ell_1/\ell_2}\;.\]
After rearranging some terms,
\[ \frac{1}{2}\|\bA \vc(\widehat\bB^\top-\bB^\star{}^\top)\|_2^2 \leq  \left\langle Y  - \bA \vc(\bB^\star{}^\top),\bA \vc(\widehat\bB^\top-\bB^\star)\right\rangle + \lambda \left(\| \bB^\star\|_{\ell_1/\ell_2}-\| \widehat\bB\|_{\ell_1/\ell_2} \right)\;.\]
Let $S_0\subset[I]$ indicate the row support of $\bB^\star$, with $|S|\leq I_0$. Write $\Delta=\widehat\bB-\bB^\star$. Then standard arguments show that
\[\| \bB^\star\|_{\ell_1/\ell_2}-\| \widehat\bB\|_{\ell_1/\ell_2} \leq \|\Delta_{S_0}\|_{\ell_1/\ell_2} - \|\Delta_{S_0^c}\|_{\ell_1/\ell_2} \;.\]
Furthermore,
\begin{align*}
\left\langle Y  - \bA \vc(\bB^\star{}^\top),\bA \vc(\widehat\bB^\top-\bB^\star{}^\top)\right\rangle 
&=\left\langle \bA^\top\left(Y  - \bA \vc(\bB^\star{}^\top)\right), \vc(\Delta^\top)\right\rangle \\
&=\sum_i \left\langle \bA_i^\top\left(Y  - \bA \vc(\bB^\star{}^\top)\right), \Delta_{i,*}\right\rangle\\
&\leq \sum_i \left( \left\| \bA_i^\top\left(Y  - \bA \vc(\bB^\star{}^\top)\right)\right\|_2\cdot \|\Delta_{i,*}\|_2\right)\\
&\leq \max_i  \left\| \bA_i^\top\left(Y  - \bA \vc(\bB^\star{}^\top)\right)\right\|_2\cdot \|\Delta\|_{\ell_1/\ell_2}\;,
\end{align*}
where $\bA_i$ is the matrix with entries $(\bA_i)_{nm,j} = \bA_{nm,ij}=X_{ni}\cdot e_j(t_{nm})$.

Consider the event $\lambda \geq 2 \max_i  \left\| \bA_i^\top\left(Y  - \bA \vc(\bB^\star{}^\top)\right)\right\|_2$. On this event, we would then have
\begin{align*}
 \frac{1}{2}\|\bA \vc(\Delta^\top)\|^2
 & \leq  \norm{\Delta}_{\ell_1/\ell_2}\cdot\max_i  \left\| \bA_i^\top\left(Y  - \bA \vc(\bB^\star{}^\top)\right)\right\|_2 + \lambda \left(\|\Delta_{S_0}\|_{\ell_1/\ell_2} - \|\Delta_{S_0^c}\|_{\ell_1/\ell_2} \right)\\ &\leq \|\Delta_{S_0}\|_{\ell_1/\ell_2}\cdot 1.5\lambda - \|\Delta_{S_0^c}\|_{\ell_1/\ell_2}\cdot 0.5\lambda \;.\end{align*}
Since the left-hand side is nonnegative, therefore, $ \|\Delta_{S_0^c}\|_{\ell_1/\ell_2}\leq 3  \|\Delta_{S_0}\|_{\ell_1/\ell_2}$. Applying the $\mathsf{RE}_{\mathsf{group}}(I_0,\alpha)$ condition, then,
\[\|\bA \vc(\Delta^\top)\|^2\geq N\alpha\norm{\Delta}^2_{\mathsf{F}}\;.\]
Combining everything,
\[\frac{1}{2}N\alpha\norm{\Delta}^2_{\mathsf{F}}\leq  \frac{1}{2}\|\bA \vc(\Delta^\top)\|^2
\leq \|\Delta_{S_0}\|_{\ell_1/\ell_2}\cdot 1.5\lambda - \|\Delta_{S_0^c}\|_{\ell_1/\ell_2}\cdot 0.5\lambda \leq \|\Delta_{S_0}\|_{\ell_1/\ell_2}\cdot 1.5\lambda \;.\] 
For the Frobenius norm result, we further write
\[ \|\Delta_{S_0}\|_{\ell_1/\ell_2}\cdot 1.5\lambda  \leq  \sqrt{|S_0|}\cdot \|\Delta_{S_0}\|_{\mathsf{F}}\cdot 1.5\lambda \leq  \sqrt{|S_0|}\cdot \|\Delta\|_{\mathsf{F}}\cdot 1.5\lambda \]
and so
\[\|\Delta\|_{\mathsf{F}}\leq \sqrt{|S_0|}\cdot 1.5\lambda\cdot \frac{2}{N\alpha}\]
while for the $\ell_1/\ell_2$-norm result we have
\[ \|\Delta\|_{\ell_1/\ell_2} =  \|\Delta_{S_0}\|_{\ell_1/\ell_2} +  \|\Delta_{S_0^c}\|_{\ell_1/\ell_2}\leq 4 \|\Delta_{S_0}\|_{\ell_1/\ell_2} \leq 4\sqrt{|S_0|} \|\Delta_{S_0}\|_{\mathsf{F}}\leq 4\sqrt{|S_0|}\|\Delta\|_{\mathsf{F}}\;.\]
This proves the theorem, as long as we can show that with high probability,
\[\lambda \geq 2 \max_i  \left\| \bA_i^\top\left(Y  - \bA \vc(\bB^\star{}^\top)\right)\right\|_2\;.\]
To prove this, take any $i\in[I]$. Then
\[
\|\bA_i^\top \left(Y  - \bA \vc(\bB^\star{}^\top)\right)\|_2
=\|\bA_i^\top \left(T + \mathbf{\epsilon}\right)\|_2 \leq \|\bA_i^\top T\|_2 + \|\bA_i^\top\mathbf{\epsilon}\|_2\;,\]
where $T$ is the truncation error defined earlier while $\mathbf\epsilon$ has entries $\mathbf{\epsilon}_{nm}=\epsilon_n(t_{nm})$. Now we bound the two pieces separately.
First, we bound the operator norm of $\bA_i$. By definition, $(\bA_i)_{nm,j}=X_{ni}\cdot e_j(t_{nm})$ and so we have
\[\bA_i^\top\bA_i = \sum_n X_{ni}^2 \bE_n^\top \bE_n \preceq \norm{X}_{\infty}^2 \sum_n \bE_n^\top \bE_n = \norm{X}_{\infty}^2 \cdot N \cdot \bF\;,\]
proving that $\|\bA_i\|_{\mathsf{op}}\leq \norm{X}_{\infty}\cdot\sqrt{N\|\bF\|_{\mathsf{op}}}$.
Then
\[ \|\bA_i^\top T\|_2 \leq  \|\bA_i\|_{\mathsf{op}}\cdot \|T\|_2\leq \norm{X}_{\infty}\cdot\sqrt{N\|\bF\|_{\mathsf{op}}}\cdot \|T\|_2\;.\]
Next, to bound
$ \|\bA_i^\top\mathbf{\epsilon}\|_2$, observe that $\bA_i^\top\mathbf{\epsilon}\sim N(0,\bA_i^\top\Sigma\bA_i)$, where
where $\Sigma$ is a block-diagonal matrix with blocks given by $\Sigma_1,\dots,\Sigma_N$.
Then, applying \cite[Proposition 1.1]{hsu2011tail}, 
\[\mathbb{P}\left\{\|\bA_i^\top\mathbf{\epsilon}\|^2_2>\trace(\bA_i^\top \Sigma\bA_i) + 2\sqrt{\trace((\bA_i^\top \Sigma\bA_i)^2)\cdot t} + 2\|\bA_i^\top \Sigma\bA_i\|_{\mathsf{op}}\cdot t\right\}\leq e^{-t}\]
for any $t\geq 0$. Since $\trace(M)\leq J\|M\|_{\mathsf{op}}$ for any $J\times J$ matrix, and using the simple identity $2\sqrt{Jt}\leq J+t$, we simplify this to the weaker statement
\[\mathbb{P}\left\{\|\bA_i^\top\mathbf{\epsilon}\|^2_2>\|\bA_i^\top \Sigma\bA_i\|_{\mathsf{op}}\cdot (2J+3t)\right\}\leq e^{-t}\]
We calculate 
\[\|\bA_i^\top \Sigma \bA_i\|_{\mathsf{op}}\leq \|\bA_i\|_{\mathsf{op}}^2 \cdot\|\Sigma\|_{\mathsf{op}} \leq \norm{X}^2_{\infty}\cdot N\|\bF\|_{\mathsf{op}}\cdot \max_n \|\Sigma_n\|_{\mathsf{op}}\;.\]
Combining everything, and setting $t=\log(I/\delta)$, for each $I$ we have
\[\mathbb{P}\left\{\|\bA_i^\top\mathbf{\epsilon}\|^2_2> \norm{X}^2_{\infty}\cdot N\|\bF\|_{\mathsf{op}}\cdot \max_n \|\Sigma_n\|_{\mathsf{op}}\cdot (2J + 3\log(I/\delta))\right\}\leq \frac{\delta}{I}\;.\]
Therefore, with probability at least $1-\delta$, for all $i=1,\dots, I$,
\[\|\bA_i^\top\mathbf{\epsilon}\|^2_2\leq \norm{X}^2_{\infty}\cdot N\|\bF\|_{\mathsf{op}}\cdot \max_n \|\Sigma_n\|_{\mathsf{op}}\cdot(2J+3\log(I/\delta))\]
and so for all $i$,
\begin{align*}
\|\bA_i^\top \left(Y  - \bA \vc(\bB^\star{}^\top)\right)\|_2
&\leq \|\bA_i^\top T\|_2 + \|\bA_i^\top\mathbf{\epsilon}\|_2\\
&\leq \norm{X}_{\infty} \sqrt{N\|\bF\|_{\mathsf{op}}}\cdot\left(\|T\|_2+\sqrt{\max_n \|\Sigma_n\|_{\mathsf{op}}\cdot(2J+3\log(I/\delta))}\right)\\& \leq \lambda/2 \;,\end{align*}
as desired.

\subsection{Proof of Theorem~\ref{t:RE_gaussian}}
We use the Transfer Principle of \citet[Lemma 5.1]{oliveira2013lower}, which connects sparse eigenvalues of $\bA$ to a restricted eigenvalue property for $\bA$. This result is stated for sparsity but can be straightforwardly extended to group-sparsity. We restate the group-sparse form here as a Lemma without proof.
\begin{lemma}[{Adapted from \citet[Lemma 5.1]{oliveira2013lower}}]\label{lem:oliviera}
Suppose $\bA$ satisfies that, for any $k_1$-row-sparse $\bW\in\mbR^{I\times J}$,
\begin{equation}\label{eqn:oliviera_assump}\norm{\bA \vc(\bW^\top) }^2_2\geq \lambda_0\norm{\bW}^2_{\mathsf{F}}\;.\end{equation}
 Then
for all $W\in\mbR^{I\times J}$,
\[\norm{\bA \vc(\bW^\top) }^2_2 \geq\lambda_0\norm{\bW}^2_{\mathsf{F}}  - \frac{\max_i\norm{\bA_i}_{\mathsf{op}}^2\cdot \norm{\bW}_{\ell_1/\ell_2}^2}{k_1-1}\;.\]
\end{lemma}
To apply this result to our work, we need to first find a $\lambda_0$ such that \eqref{eqn:oliviera_assump} is satisfied (with high probability), and then we need to compute a bound (holding with high probability) on each $\norm{\bA_i}_{\mathsf{op}}$.

\paragraph{Step 1: finding $\lambda_0$ for condition \eqref{eqn:oliviera_assump}}

We first give another lemma, proved below:
\begin{lemma}\label{lem:RIP}
Suppose that the assumptions of Theorem~\ref{t:RE_gaussian} hold, and take any fixed sequence $\{\bE_n\}$ satisfying \eqref{e:Econd}. Choose any $\delta>0$ and any $k_1\geq 1$. Then
there are constants $a_1,a_2>0$ depending only on $\nu,\lambda_{\min}(\Sigma),\gamma_0,\gamma_1$, such that if 
\begin{equation}\label{eqn:N_RIP}N \geq a_1 (k_1 J+k_1\log(I)+\log(1/\delta)),\end{equation}
then with probability at least $1-\delta/2$, for all $k_1$-row-sparse $\bW$,
\[\frac{\sum_n\norm{\bE_n^\top \bW^\top X_n}_2}{N\sqrt{\trace(\bW^\top\Sigma \bW)}} \geq a_2\;.\]
\end{lemma}

To apply these results to our work, first note that 
by Lemma~\ref{lem:RIP}, if $N$ satisfies \eqref{eqn:N_RIP}, then with probability at least $1-\delta/2$, for all $k_1$-sparse $\bW$ (where we specify $k_1$ later), 
\begin{align*}
\norm{A\cdot\mathsf{vec}(\bW)}^2_2  = \sum_n\norm{\bE_n^\top \bW^\top X_n}^2_2 & \geq \frac{1}{N}\left(\sum_n\norm{\bE_n^\top \bW^\top X_n}_2\right)^2 \\
& \geq N \trace(\bW^\top\Sigma \bW)\cdot (a_2)^2\;.
\end{align*}
In particular, the assumption \eqref{eqn:oliviera_assump} in Lemma~\ref{lem:oliviera} holds with $\lambda_0 = N (a_2)^2\lambda_{\min}(\Sigma)$.

\paragraph{Step 2: bounding $\norm{\bA_i}_{\mathsf{op}}$}
Next we bound $\max_i\norm{\bA_i}^2_{\mathsf{op}}$. From the proof of Theorem~\ref{thm:main_sparse}, we know that
$\norm{\bA_i}^2_{\mathsf{op}}\leq \norm{X}^2_{\infty}\cdot N\norm{\bF}_{\mathsf{op}}$, where $\bF=\frac{1}{N}\sum_n \bE_n^\top\bE_n$. By \eqref{e:Econd}, we see that $\norm{\bF}_{\mathsf{op}}\leq \gamma_1^2$.
Furthermore, $X_{ij}$ is mean zero and $\nu$-subgaussian, and so, by standard subgaussian tail bounds, with probability at least $1-\delta/2$,
\[\norm{X}_{\infty}=\max_{ni}|X_{ni}|\leq \nu\sqrt{2\log(4IJ/\delta)}\;.\]

\paragraph{Step 3: applying Lemma~\ref{lem:oliviera}}
Combining the results from Step 1 and Step 2, and applying Lemma~\ref{lem:oliviera}, we see that with probability at least $1-\delta$, for all $\bW\in\mbR^{I\times J}$,
\[\norm{\bA \vc(\bW^\top) }^2_2 \geq N\left[(a_1)^2\lambda_{\min}(\Sigma)\cdot \norm{\bW}^2_{\mathsf{F}}  - \frac{2\nu^2\log(4IJ/\delta)\cdot \gamma_1^2\cdot \norm{\bW}_{\ell_1/\ell_2}^2}{k_1-1}\right]\;.\]
For any $\bW$ with $\norm{\bW_{S^c}}_{\ell_1/\ell_2}\leq 3\norm{\bW_S}_{\ell_1/\ell_2}$ (where $|S|\leq k$), we then have
\[\norm{\bW}_{\ell_1/\ell_2} \leq 4\norm{\bW_S}_{\ell_1/\ell_2}\leq 4\sqrt{k}\norm{\bW_S}_{\mathsf{F}}\leq 4\sqrt{k}\norm{\bW}_{\mathsf{F}}\;,\]
and so the result above gives
\[\norm{\bA \vc(\bW^\top) }^2_2 \geq N\norm{\bW}^2_{\mathsf{F}}\left[(a_2)^2\lambda_{\min}(\Sigma)  - \frac{2\nu^2\log(4IJ/\delta)\cdot \gamma_1^2\cdot 16k}{k_1-1}\right]\;.\]
Taking
\[k_1 \geq 1 + \frac{2\nu^2\log(4IJ/\delta)\cdot \gamma_1^2\cdot 16k}{0.5(a_2)^2\lambda_{\min}(\Sigma)}\;,\]
we obtain
\[\norm{\bA \vc(\bW^\top) }^2_2 \geq N\norm{\bW}^2_{\mathsf{F}}\cdot 0.5(a_2)^2\lambda_{\min}(\Sigma)\]
for all $\bW\in\mbR^{I\times J}$ with $\norm{\bW_{S^c}}_{\ell_1/\ell_2}\leq 3\norm{\bW_S}_{\ell_1/\ell_2}$  for any $|S|\leq k$, as desired. The sample size requirement \eqref{eqn:N_RIP_thm} in Theorem~\ref{t:RE_gaussian} ensures that the sample size assumption in Lemma~\ref{lem:RIP} will hold for the specified choice of $k_1$.

\subsubsection{Proof of Lemma~\ref{lem:RIP}}

For the proof of this lemma, we'll need several supporting lemmas, proved below.
The first lemma shows that $\sum_n \norm{\bE_n^\top \bW^\top X_n}_2$ satisfies
upper and lower bounds with high probability for any fixed {\em single} matrix $\bW$.
\begin{lemma}\label{lem:Wfixed_highprob}
For any fixed $\bW$,
\[\mathbb{P}\left\{c_1 \leq \frac{\sum_n \norm{\bE_n^\top \bW^\top X_n}_2}{N\sqrt{\trace(\bW^\top\Sigma \bW)}}\leq c_2\right\} \geq 1 - 2e^{-c_3N}\;,\]
where $c_1,c_2,c_3>0$ depend only on $\nu,\lambda_{\min}(\Sigma),\gamma_0,\gamma_1$.
\end{lemma}
The next lemma shows that proving a lower bound on $\sum_n\norm{\bE_n^\top \bW^\top X_n}_2$ over all row-sparse $\bW$ can be reduced to proving a lower bound on a finite covering set. This result is a simple extension of \citet[Lemma 5.1]{baraniuk2008simple}, from the sparse setting to the group-sparse setting, and so we do not give the proof here.
\begin{lemma}[{Adapted from \citet[Lemma 5.1]{baraniuk2008simple}}]\label{lem:covering}
Let
\[\mathcal{W} = \left\{\bW\in\mbR^{I\times J}:\text{ $\bW$ is $k_1$-row-sparse}, \ \trace(\bW^\top\Sigma \bW)=1\right\}\;.\]
Let $f(\bW)$ be any seminorm (that is, $f(c\cdot \bW)=|c|\cdot f(\bW)$ and $f(\bW_1+\bW_2)\leq f(\bW_1)+f(\bW_2)$). Choose any $\epsilon\in(0,1/2)$. Then there exists $\mathcal{W}'\subset \mathcal{W}$ with $|\mathcal{W}'|\leq I^{k_1}(3/\epsilon)^{k_1J}$ such that
\[\inf_{\bW\in\mathcal{W}}f(\bW)\geq \inf_{\bW\in\mathcal{W}'}f(\bW) - \frac{\epsilon}{1-\epsilon}\sup_{\bW\in\mathcal{W}'}f(\bW)\;.\]
\end{lemma}

\begin{proof}[Proof of Lemma~\ref{lem:RIP}]

By Lemma~\ref{lem:Wfixed_highprob}, for any fixed $\bW$,
with probability at least $1-2e^{-c_3N}$,
\[c_1 \leq \frac{\sum_n \norm{\bE_n^\top \bW^\top X_n}_2}{N\sqrt{\trace(\bW^\top\Sigma \bW)}}\leq c_2\]
where $c_1,c_2,c_3>0$ depend only on $\nu,\lambda_{\min}(\Sigma),\gamma_0,\gamma_1$.
Now choose $\epsilon =\frac{c_1}{4c_2}$.
In the notation of Lemma~\ref{lem:covering}, let $f(\bW)=\sum_n \norm{\bE_n^\top \bW^\top X_n}_2$, and take $\mathcal{W}'$ as in the lemma. 
Since  $|\mathcal{W}'|\leq I^{k_1}(3/\epsilon)^{k_1J}$, the above bound is true for all $\bW\in\mathcal{W}'$ with probability at least
\[1-2e^{-c_3N + \log( I^{k_1}(3/\epsilon)^{k_1J})}\geq 1-\delta/2,\]
by our lower bound  $N\geq a_1 (k_1 J+k_1\log(I)+\log(1/\delta))$ as long as the constant $a_1$ is chosen appropriately.
Next, applying Lemma~\ref{lem:covering}, for all $\bW\in\mathcal{W}$,
\[\frac{ \sum_n\norm{\bE_n^\top \bW^\top X_n}_2}{N\sqrt{\trace(\bW^\top\Sigma \bW)}}\geq c_1 - \frac{\epsilon}{1-\epsilon}\cdot c_2 \geq \frac{c_1}{2}.\]
After defining $a_2=c_1/2$, this proves the lemma.
\end{proof}

\subsubsection{Proof of Lemma~\ref{lem:Wfixed_highprob}}
We begin by stating a supporting lemma.
\newcommand{\EE}[1]{\mathbb{E}\left[{#1}\right]}
\newcommand{\PP}[1]{\mathbb{P}\left\{{#1}\right\}}
\newcommand{\One}[1]{\mathbf{1}\left\{{#1}\right\}}
\begin{lemma}\label{lem:subgaussian_norm_sum}
Let $B_1\in\mathbb{R}^{M_1\times d},\dots,B_N\in\mathbb{R}^{M_N\times d}$ be fixed matrices with
\[\sum_{n=1}^N \norm{B_n}_{\mathsf{F}} \geq c_1 N\]
and
\[\sqrt{\sum_{n=1}^N \norm{B_n}^2_{\mathsf{F}}} \leq c_2 \sqrt{N}\]
for some $c_1,c_2>0$.
Let $X_1,\dots,X_N\in\mathbb{R}^d$ be iid~random vectors with $\EE{X_n}=0$ and $\EE{X_nX_n^\top}=\mathbb{I}_d$,
which are $\sigma$-subgaussian, meaning that $\EE{e^{v^\top X_n}}\leq e^{\sigma^2\norm{v}^2_2/2}$ for all fixed $v\in\mathbb{R}^d$.
Then
\[\PP{c_3 N \leq \sum_{n=1}^N \norm{B_n X_n}_2 \leq c_4 N}\geq 1 - 2e^{-c_5N},\]
where $c_3,c_4,c_5>0$ depend only on $c_1,c_2,\sigma$.
\end{lemma}

\begin{proof}[Proof of Lemma~\ref{lem:Wfixed_highprob}]
First we calculate
\begin{multline*}
\sum_n {\norm{\bE_n^\top \bW^\top \Sigma^{1/2}}_{\mathsf{F}}^2} 
  = \sum_n {\trace(\bE_n^\top \bW^\top \Sigma \bW \bE_n)}\\
 =  {\trace(\Sigma^{1/2}\bW(\sum_n \bE_n\bE_n^\top)\bW^\top \Sigma^{1/2})}
 \leq  {\trace(\Sigma^{1/2}\bW\bW^\top \Sigma^{1/2})}\cdot N\norm{\frac{1}{N}\sum_n \bE_n\bE_n^\top}\\
 =  {\trace(\bW^\top\Sigma \bW)}\cdot N\norm{\frac{1}{N}\sum_n \bE_n\bE_n^\top}
 \leq \gamma_1 N  {\trace(\bW^\top\Sigma \bW)}\;.
\end{multline*}
And,
\begin{multline*}\sum_n\norm{\bE_n^\top \bW^\top \Sigma^{1/2}}_{\mathsf{F}} 
 =\sum_n \sqrt{\sum_i \norm{(\Sigma^{1/2}\bW)_{i,*}\bE_n}^2_2}
 \geq\sum_n \sum_i \frac{\norm{(\Sigma^{1/2}\bW)_{i,*}}_2}{\sqrt{\trace(\bW^\top\Sigma \bW)}}\norm{(\Sigma^{1/2}\bW)_{i,*}\bE_n}_2\\
 \geq \gamma_0 N\sum_i \frac{\norm{(\Sigma^{1/2}\bW)_{i,*}}^2_2}{\sqrt{\trace(\bW^\top\Sigma \bW)}}
=\gamma_0 N\sqrt{\trace(\bW^\top\Sigma \bW)}\;.
\end{multline*}

Now, writing $\widetilde{X}_n = \Sigma^{-1/2}X_n$ and defining
\[B_n = \frac{\bE_n^\top \bW^\top}{\sqrt{\trace(\bW^\top \Sigma \bW)}},\]
 we see that the $\widetilde{X}_n$'s are
iid~with mean zero, identity covariance, and are $\nu/\sqrt{\lambda_{\min}(\Sigma)}$-subgaussian.
And, $\norm{\bE_n^\top \bW^\top {X}_n}_2 = \sqrt{\trace(\bW^\top \Sigma \bW)}\cdot \norm{B_n  \Sigma^{1/2}\widetilde{X}_n}_2$. 
Applying Lemma~\ref{lem:subgaussian_norm_sum}, then, for some constants $c_1,c_2,c_3$ depending only on $\nu,\lambda_{\min}(\Sigma),\gamma_0,\gamma_1$,
\[\mathbb{P}\left\{ c_1N\sqrt{\trace(\bW^\top \Sigma \bW)}\leq \sum_n \norm{\bE_n^\top \bW^\top {X}_n}_2\leq c_2N\sqrt{\trace(\bW^\top \Sigma \bW)}\right\}\geq 1 - 2e^{-c_3N}.\]
\end{proof}

\subsubsection{Proof of Lemma~\ref{lem:subgaussian_norm_sum}}
Before we prove Lemma~\ref{lem:subgaussian_norm_sum}, we first state two additional
supporting lemmas, proved below:
\begin{lemma}\label{lem:proportion_interval}
Let $x\in\mathbb{R}^N$ be a vector with $ \norm{x}_1\geq c_1N$, $\norm{x}_2\leq c_2\sqrt{N}$. Then
\[\left|\left\{n:  |x_n|\geq \frac{c_1}{2}\right\}\right| \geq \frac{(c_1)^2}{4(c_2)^2} N.\]
\end{lemma}

\begin{lemma}\label{lem:nonneg_subg}
Suppose $Z_1,\dots,Z_N$ are independent non-negative random variables
with $\EE{(Z_n)^2}\geq a_1$ and $\EE{e^{(Z_n)^2}}\leq a_2$ for all $n=1,\dots,N$, where $a_1,a_2>0$.
Then
\[\PP{ \sum_n Z_n < a_3 N}\leq e^{-a_4 N}\text{\quad and \quad}
\PP{\sum_n (Z_n)^2> a_5 N}\leq e^{-a_6 N}\]
where $a_3,a_4,a_5,a_6>0$ are constants depending only on $a_1,a_2$.
\end{lemma}

Now we prove Lemma~\ref{lem:subgaussian_norm_sum}.
\begin{proof}[Proof of Lemma~\ref{lem:subgaussian_norm_sum}]
For each $n$, define
$Z_n = \frac{\norm{B_n X_n}_2}{2\sigma \norm{B_n}_{\mathsf{F}}}$.
We have
\[\EE{\norm{B_n X_n}^2_2} = \EE{X_n^\top B_n^\top B_n X_n} = \EE{\trace(B_n^\top B_n X_n X_n^\top)} = \trace(B_n^\top B_n) = \norm{B_n}^2_{\mathsf{F}},\]
and so
$\EE{(Z_n)^2} = \frac{1}{4\sigma^2}=: a_1$.
Next, it is known from \citet[Remark 2.3]{hsu2011tail} that
\[\EE{e^{\eta\norm{B_n X_n}^2_2} }\leq \sigma^2\trace(B_n^\top B_n) \eta + \frac{\sigma^4\trace((B_n^\top B_n)^2)\eta^2}{1 - 2\sigma^2\norm{B_n^\top B_n}\eta}\]
for all $0\leq \eta < \frac{1}{2\sigma^2\norm{B_n^\top B_n}}$. Using basic facts about matrix norms we have $\norm{B_n^\top B_n}=\norm{B_n}^2 \leq \norm{B_n}^2_{\mathsf{F}}$ and $\trace(B_n^\top B_n) = \norm{B_n}^2_{\mathsf{F}}$ and 
$\trace((B_n^\top B_n)^2) = \norm{(B_n^\top B_n)^2}^2_{\mathsf{F}} \leq \norm{B_n^\top B_n}^4_{\mathsf{F}}$. Setting $\eta = \frac{1}{4\sigma^2\norm{B_n}^2_{\mathsf{F}}}$ we obtain
\[\EE{e^{(Z_n)^2}} = \EE{e^{\norm{B_n X_n}^2_2/(4\sigma^2\norm{B_n}^2_{\mathsf{F}})}} \leq e^{3/8}=: a_2.\]
Applying Lemma~\ref{lem:nonneg_subg} to $\{Z_n:n=1,\dots,n\}$, we see that for constants $a_5,a_6>0$ depending only on $a_1,a_2$,
\[\PP{ \sum_n  \frac{\norm{B_n X_n}^2_2}{4\sigma^2 \norm{B_n}^2_{\mathsf{F}}} \leq a_5 N}\geq 1 - e^{-a_6 N}.\]

Furthermore, setting
$\mathcal{I} = \{n : \norm{B_n}_{\mathsf{F}} \geq c_1/2\}$,
by applying Lemma~\ref{lem:proportion_interval} to the vector $(\norm{B_n}_{\mathsf{F}})_{n=1,\dots,N}$, 
we see that 
$\left|\mathcal{I}\right| \geq \frac{(c_1)^2}{4(c_2)^2}  N$.
Applying Lemma~\ref{lem:nonneg_subg} to $\{Z_n:n\in\mathcal{I}\}$
 we see that for constants $a_3,a_4>0$ depending only on $a_1,a_2$,
\[\PP{ \sum_{n\in\mathcal{I}}  \frac{\norm{B_n X_n}_2}{2\sigma \norm{B_n}_{\mathsf{F}}} \geq a_3 |\mathcal{I}|}\geq 1 - e^{-a_4 |\mathcal{I}|}\geq 1 - e^{-a_4 \frac{(c_1)^2}{4(c_2)^2}  N}.\]

Next, assume these events hold. We have
\[
\sum_n \norm{B_nX_n}_2 
=\sum_n \frac{\norm{B_nX_n}_2}{\norm{B_n}_{\mathsf{F}}} \cdot \norm{B_n}_{\mathsf{F}}\\
\leq \sqrt{\sum_n \frac{\norm{B_n X_n}^2_2}{\norm{B_n}^2_{\mathsf{F}}}}\cdot\sqrt{\sum_n \norm{B_n}^2_{\mathsf{F}}}\\
\leq 2\sigma \sqrt{a_5 N}\cdot c_2\sqrt{N}.
\]
Furthermore,
\begin{multline*}
\sum_n \norm{B_nX_n}_2 
\geq \sum_{n\in\mathcal{I}} \norm{B_n X_n}_2
= \sum_{n\in\mathcal{I}} 2\sigma\norm{B_n}_{\mathsf{F}} \cdot\frac{\norm{B_n X_n}_2}{2\sigma\norm{B_n}_{\mathsf{F}}}\\
\geq c_1 \cdot \sum_{n\in\mathcal{I}} \frac{\norm{B_n X_n}_2}{2\sigma\norm{B_n}_{\mathsf{F}}}
\geq c_1 \cdot a_3|\mathcal{I}|
\geq c_1 \cdot a_3\cdot \frac{(c_1)^2}{4(c_2)^2}  N\;.
\end{multline*}

Setting
\[c_3 =  c_1 \cdot a_3\cdot \frac{(c_1)^2}{4(c_2)^2}, \ c_4 =2\sigma c_2\sqrt{a_5} , \ c_5 = \min\left\{ a_6, a_4 \frac{(c_1)^2}{4(c_2)^2} \right\},\]
we have proved the lemma.
\end{proof}

Finally, we prove the two supporting results, Lemmas~\ref{lem:proportion_interval} and~\ref{lem:nonneg_subg}.

\begin{proof}[Proof of Lemma~\ref{lem:proportion_interval}] 
We have
\begin{align*}
c_1N
&\leq \sum_{n=1}^N |x_n|
= \sum_{n=1}^N |x_n|\cdot \One{|x_n|\geq c_1/2} +  \sum_{n=1}^N |x_n|\cdot \One{|x_n|< c_1/2} \\
&\leq \sum_{n=1}^N |x_n|\cdot \One{|x_n|\geq c_1/2} + N \cdot c_1/2\\
&\leq \sqrt{\sum_{n=1}^N |x_n|^2}\cdot \sqrt{\sum_{n=1}^N\One{|x_n|\geq c_1/2}^2} + N \cdot c_1/2\\
&\leq c_2\sqrt{N}\cdot \sqrt{\left|\left\{n:  |x_n|\geq c_1/2\right\}\right| } + N \cdot c_1/2,
\end{align*}
and therefore,
$\left|\left\{n: |x_n|\geq c_1/2\right\}\right| \geq \frac{(c_1)^2}{4(c_2)^2} N$.
\end{proof}

\begin{proof}[Proof of Lemma~\ref{lem:nonneg_subg}]
For the upper bound, we have
\[\PP{\sum_n (Z_n)^2 > 2a_2 N} \leq \EE{e^{\sum_n (Z_n)^2 - 2a_2N}}  = \left(\prod_n \EE{e^{(Z_n)^2}}\right)\cdot e^{-2a_2N} = e^{-a_2N}.\]
For the lower bound, we have, for each $n$,
\begin{multline*}
a_1
\leq \EE{(Z_n)^2}
= \EE{(Z_n)^{0.5}(Z_n)^{1.5}}
\leq \sqrt{\EE{Z_n}}\cdot\sqrt{\EE{(Z_n)^3}}\text{\quad since $Z_n\geq 0$}\\
\leq \sqrt{\EE{Z_n}}\cdot\sqrt[4]{\EE{(Z_n)^6}}
\leq \sqrt{\EE{Z_n}}\cdot\sqrt[4]{\EE{6e^{(Z_n)^2} }}
\leq \sqrt{\EE{Z_n}}\cdot\sqrt[4]{6a_2}
\end{multline*}
and so
$\EE{Z_n}\geq \frac{a_1^2}{\sqrt{6a_2}}$.
Next, taking any $t\in[0,1]$,
\begin{align*}
\EE{e^{-tZ_n}} 
&= 1 - t\EE{Z_n} + \sum_{k\geq 2} \frac{(-1)^kt^k\EE{(Z_n)^k}}{k!}\\
&\leq -t\EE{Z_n} + \left(1 + \sum_{k\geq 1} \frac{t^{2k} \EE{(Z_n)^{2k}}}{(2k)!}\right)\text{\quad by removing negative terms}\\
&= -t\EE{Z_n} + \EE{e^{t^2(Z_n)^2}}\text{\quad since $k!\leq (2k)!$}\\
&\leq -t\EE{Z_n} + \EE{t^2 e^{(Z_n)^2} + (1-t^2)\cdot 1}\text{\quad by convexity of $z\mapsto e^z$}\\
&\leq 1- t \frac{a_1^2}{\sqrt{6a_2}} + t^2 (a_2-1).
\end{align*}

Setting $t = \min\left\{1, \frac{a_1^2}{2(a_2-1)\sqrt{6a_2}}\right\}$ we get
\[\EE{e^{-tZ_n}}\leq 1- t\frac{(a_1)^2}{2\sqrt{6a_2}}\leq e^{-t \frac{(a_1)^2}{2\sqrt{6a_2}}}.\]
Then
\begin{multline*}\PP{\sum_n Z_n < \frac{(a_1)^2}{4\sqrt{6a_2}} N} \leq \EE{e^{t N \frac{(a_1)^2}{4\sqrt{6a_2}} - t \sum_n Z_n}} \\= \EE{e^{tN \frac{(a_1)^2}{4\sqrt{6a_2}}}} \left(\prod_n \EE{e^{-tZ_n}}\right) = e^{t N \frac{(a_1)^2}{4\sqrt{6a_2}}- t N \frac{(a_1)^2}{2\sqrt{6a_2}}} = e^{- t N \frac{(a_1)^2}{4\sqrt{6a_2}}}.\end{multline*}
Setting
\[a_3 =  \frac{(a_1)^2}{4\sqrt{6a_2}} , \ a_4 =  \min\left\{1, \frac{a_1^2}{2(a_2-1)\sqrt{6a_2}}\right\}\cdot  \frac{(a_1)^2}{4\sqrt{6a_2}}, \ a_5= 2a_2, \ a_6 = a_2,\]
we have proved the lemma.
\end{proof}

\section{Proofs for Section \ref{s:func}}
We begin this section with the proof of Theorem \ref{t:res}.  After, we will divide the proof of Theorem \ref{t:highfreq_main} into a sequence of lemmas.

\subsection{Proof of Theorem  \ref{t:res}}
Let $e_1, e_2, \dots$ be an orthonormal basis of $\mcH$.  Then each coordinate $x_i$ of $x\in \mcH^I$, can be expressed as
\[
x_i = \sum_k x_{i,k} e_k.
\]
We denote $x_{i,k} \in \mbR$ as the coordinates with respect to the $e_1,e_2,\dots$ basis, and we will also let $x^{(k)}$ represent the vector $\{x_{1,k}, \dots, \edits{ x_{I,k}}\}$.  We can then write
\begin{align*}
 \| \bX  x  \|_{\mcH}^2 
 & =  \langle \bX x, \bX x \rangle_{\mcH} = \sum_{n=1}^N \langle X_n x, X_n x \rangle  \\
 & =   \sum_{n=1}^N \sum_{i=1}^I \sum_{j=1}^I X_{ni} X_{nj} \langle x_i ,x_j\rangle_{\mcH}\\
 \intertext{Expressing the $x_i$ with respect to the orthonormal basis $e_1,e_2,\dots$,}
&= \sum_n \sum_i \sum_j \sum_k \sum_l  X_{ni} X_{nj} x_{i,k} x_{j,l} \langle e_k, e_l \rangle \\
 &=  \sum_n \sum_i \sum_j \sum_k  X_{ni} X_{nj} x_{i,k} x_{j,k}  \\
 & = \sum_k  x^{(k) \top} \bX^\top \bX  x^{(k)}  \\
  & =\sum_k \| \bX x^{(k)}\|_2^2 .
  \end{align*}
  Note that the last norm above is simply the Euclidean norm on $\mbR^N$.  By assumption in the Theorem, we have that the above is bounded from below as follows:
  \begin{align*}
 \| \bX  x  \|_{\mcH}^2 &=   \sum_k \| \bX x^{(k)}\|_2^2 \\
 &\geq\sum_k \left(c_1 \sqrt{N } \|x^{(k)}\|_2  - c_2\sqrt{\log(I)} \|x^{(k)}\|_{1} \right)^2 \\
  & = (c_1)^2 N \sum_k\|x^{(k)}\|_2^2
  -2 c_1 c_2 \sqrt{N \log( I)} \sum_k  \|x^{(k)}\|_2 \|x^{(k)}\|_1
  +  (c_2)^2  \log( I)  \sum_k \|x^{(k)}\|_{1}^2. 
 \end{align*}
 Examining the first sum in this last line, we have by Parceval's identity
 \[
  \|x\|_{\mcH}^2 = \sum_i \|x_i\|_{\mcH}^2 = \sum_i \sum_k | x_{i,k}|^2 = \sum_{k} \| x^{(k)}\|^2.
 \]
For the second sum, the Cauchy-Schwartz inequality gives
 \begin{align*}
  \sum_k  \|x^{(k)}\|_2 \|x^{(k)}\|_1 
  \leq  \| x\|_{\mcH} \sqrt{ \sum_k \|x^{(k)}\|_1^2}.
 \end{align*}
So we have that 
\[
 \| \bX  x  \|_{\mcH}^2 
 \geq \left(c_1 \sqrt{N} \|x\|_{\mcH} - c_2 \sqrt{  \log( I)}\sqrt{ \sum_k \|x^{(k)}\|_1^2}\right)^2
\]
or equivalently
\[
\| \bX  x  \|_{\mcH} \geq c_1 \sqrt{N} \|x\|_{\mcH} - c_2 \sqrt{  \log( I)}\sqrt{ \sum_k \|x^{(k)}\|_1^2}.
\]
Now for the last piece observe that (by Cauchy-Schwarz)
\begin{align*}
\sum_k \|x^{(k)}\|_1^2 & = \sum_k \sum_i \sum_j |x_{i,k}| |x_{j,k}| \\
& \leq \sum_i \sum_j \left(\sum_k |x_{i,k}|^2\right)^{1/2} \left(\sum_k |x_{j,k}|^2\right)^{1/2} \\
& = \left(\sum_i  \left(\sum_k |x_{i,k}|^2\right)^{1/2}  \right)^2  \\
& =  \left(\sum_i  \| x_i\|_{\mcH}  \right)^2 = \| x\|_{\ell_1/\mcH}^2.
\end{align*}
We can therefore conclude that
\[
\| \bX  x  \|_{\mcH} \geq c_1 \sqrt{N} \|x\|_{\mcH} - c_2 \sqrt{\log(I)}\| x\|_{\ell_1/\mcH}
\]
as claimed.

{\color{black}
Next suppose that $c_1>4c_2\sqrt{\frac{I_0\log(I)}{N}}$. We will show that $\bX$ satisfies the $\mathsf{RE}_{\mathsf{F}}(I_0,\alpha)$ property with $\alpha = \left(c_1 - 4c_2\sqrt{\frac{I_0\log(I)}{N}}\right)^2$.
To see this,  take any set $S$ with $|S|\leq I_0$, and any $x$ with $\norm{x_{S^c}}_{\ell_1/\mathcal{H}}\leq 3\norm{x_{S}}_{\ell_1/\mathcal{H}}$. By the Cauchy-Schwarz inequality, which relates the $\ell_1$ and $\ell_2$ norms of any vector (including a vector of functions $x_S$),
we have
\[\norm{x_S}_{\ell_1/\mathcal{H}}\leq \norm{x_S}_{\mathcal{H}}\cdot \sqrt{|\text{Support}(x_S)|} \leq \norm{x_S}_{\mathcal{H}}\cdot\sqrt{I_0}.\]
By the triangle inequality,
\[\norm{x}_{\ell_1/\mathcal{H}} \leq \norm{x_S}_{\ell_1/\mathcal{H}}  + \norm{x_{S^c}}_{\ell_1/\mathcal{H}}  \leq 4\norm{x_S}_{\ell_1/\mathcal{H}}  \leq 4\norm{x_S}_{\mathcal{H}}\cdot\sqrt{I_0} \leq 4\norm{x}_{\mathcal{H}}\cdot\sqrt{I_0}.\]
Then returning to the above,
\begin{multline*}
\| \bX  x  \|_{\mcH} \geq c_1 \sqrt{N} \|x\|_{\mcH} - c_2 \sqrt{\log(I)}\| x\|_{\ell_1/\mcH}
\geq  c_1 \sqrt{N} \|x\|_{\mcH} - c_2 \sqrt{\log(I)}\left( 4\norm{x}_{\mathcal{H}}\cdot\sqrt{I_0}\right)\\
 = \norm{x}_{\mcH} \cdot \sqrt{N}\cdot \left(c_1 - 4c_2 \sqrt{\frac{I_0\log(I)}{N}}\right) = \norm{x}_{\mcH}\cdot \sqrt{\alpha N},
\end{multline*}
proving that $\bX$ satisfies the $\mathsf{RE}_{\mathsf{F}}(I_0,\alpha)$ property.
}

\subsection{Proof of Theorem \ref{t:highfreq_main}}
We follow a very similar structure to the one found in Chapter 6 of \citet{buhlmann:vandegeer:2011}, adapting the arguments for function spaces as needed.  We break the proof into a sequence of Lemmas which can be thought of as a basic inequality, a functional concentration inequality, applying that inequality, a last lemma to set up the final proof.  We then combine all the lemmas to prove Theorem  \ref{t:highfreq_main}.

\begin{lemma}[Basic Inequality] \label{l:basic_ineq}
With probability one, we always have the inequality
\[
\frac{1}{2} \| \bX(\widehat \beta-\beta^\star)\|_{\mcH}^2 + \lambda \| \widehat \beta\|_{\ell_1/\mcH}
\leq \left( \max_{1 \leq i \leq I} \|\eg^\top \bX^{(i)}\|  \right) \| \widehat \beta - \beta^\star\|_{\ell_1/\mcH} + \lambda \| \beta^\star \|_{\ell_1/\mcH}.
\]
\begin{proof}
We can set up a \textit{basic inequality} using
\[
\frac{1}{2} \| Y  - \bX \widehat \beta\|^2 +\lambda  \| \widehat \beta \|_{\ell_1/\mcH} 
\leq \frac{1}{2} \| Y  - \bX \beta^\star\|^2 + \lambda  \| \beta^\star \|_{\ell_1/\mcH}. 
\]
We can rewrite
\begin{align*}
\frac{1}{2} \| Y  - \bX \widehat \beta\|^2 
& = \frac{1}{2} \| Y  - \bX \beta^\star - \bX(\widehat \beta- \beta^\star)\|^2 \\
& = \frac{1}{2} \| Y  - \bX \beta^\star \|^2 + \frac{1}{2} \| \bX(\widehat \beta-\beta^\star)\|^2-  \frac{2}{2}\langle Y  - \bX \beta^\star,  \bX(\widehat \beta-\beta^\star)\rangle.
\end{align*}
This gives
\begin{align*}
\frac{1}{2} \| \bX(\widehat \beta-\beta^\star)\|^2 + \lambda \| \widehat \beta\|_{\ell_1/\mcH}
 & \leq  \langle Y  - \bX \beta^\star,  \bX(\widehat \beta-\beta^\star)\rangle +  \lambda \|  \beta ^\star\|_{\ell_1/\mcH} \\
 & = \sum_{i=1}^I \langle  \eg^\top \bX^{(i)},  \widehat \beta_i-\beta^\star_i \rangle +  \lambda \|  \beta ^\star\|_{\ell_1/\mcH} \\
 & \leq \sum_{i=1}^I  \| \eg^\top \bX^{(i)}\| \|  \widehat \beta_i-\beta^\star_i \| +  \lambda \|  \beta ^\star\|_{\ell_1/\mcH} \\
 & \leq   \max_{1 \leq i \leq I}  \| \eg^\top \bX^{(i)}\| \|  \widehat \beta-\beta^\star \|_{\ell_1/\mcH} +  \lambda \|  \beta ^\star\|_{\ell_1/\mcH},
\end{align*}
which is the desired result.
\end{proof}
\end{lemma}

\begin{lemma} \label{l:n_exp_ineq}
Let $X$ be an $\mcH$ valued Gaussian process with mean zero and covariance operator $C$.  Let $\Lambda^\top= (\lambda_1, \lambda_2, \dots)$ be a vector of the eigenvalues of $C$ (in decreasing order).  Then we have the bound
\[
P\left\{   
 \|  X\|^2 \geq \| \Lambda\|_1  + 2 \|\Lambda\|_2  \sqrt{  t } + 2  \|\Lambda\|_\infty t
\right\}
\leq \exp(-t).
\]
\begin{proof}
Using the Karhunen-Loeve expansion, we can  express
\[
\|X\|^2 \overset{\mcD}{=} \sum_{j=1}^\infty \lambda_j Z_j^2,
\]
where $\{Z_j\}$ are iid standard normal.  Note that since $X$ is square integrable, we have that $\sum_i \lambda_i  = \|\Lambda\|_1< \infty$, and therefore $\|\Lambda\|_2 < \infty $ as well.  Define the events, for $J = 1, 2, \dots$
\[
\mcA_J = \left\{  \sum_{j=1}^J \lambda_j Z_j^2 \geq \| \Lambda\|_1  + 2 \|\Lambda\|_2  \sqrt{  t } + 2 \lambda_1 t \right\}.
\]
Since $\| \Lambda\|_1  \geq \sum_{j=1}^J \lambda_j $ and $\|\Lambda\|_2^2 \geq  \sum_{j=1}^J \lambda_j^2$, it follows from Lemma 1 in \citet{laurent:massart:2000} that
\[
P(\mcA_J) \leq \exp(-t),
\]
for all $J$.  Since $ \sum_{j=1}^J \lambda_j Z_j^2$ is a strictly increasing sequence of $J$ we have that
\[
\mcA_1 \subseteq \mcA_2 \subseteq \mcA_3 \dots 
\]
Using the continuity from below of probability measures, we can pass the limit $J \to \infty$ to obtain the desired result.
\end{proof}
\end{lemma}

\begin{lemma}\label{l:F}
Assume that $\eg = (\eg_1, \dots, \eg_N)$ are iid~Gaussian elements of $\mcH$ and $\bX$ is a design matrix with standardized columns.  Define the event
\[
\mathscr{F} = \left\{ \max_{1 \leq i \leq I}  \|\eg^\top \bX^{(i)} \|   \leq \lambda_0 \right\}
\]
where, for $\delta > 0$, 
\[
\lambda_0 = \sqrt{N}  \sqrt{ \| \Lambda\|_1  + 2 \|\Lambda\|_2 \sqrt{\log(I/\delta)}   + 2 \lambda_1  \log(I/\delta)}.
\]
Then we have that
\[
P( \mathscr{F})  \geq 1 - \delta.
\]

\begin{proof}
Notice that if $\{\eg_n\}$ are iid Gaussian with mean zero and covariance operator $C$, then $\eg^\top \bX^{(i)}/\sqrt{N}$ is also mean zero Gaussian, and the covariance operator is given by
\[
\frac{1}{N }\Cov(\eg^\top \bX^{(i)}) = \frac{1}{N}\sum_{n=1}^N \Cov(\eg_n X_{ni}) = C,
\]
since the columns of $\bX$ are standardized.  Therefore, $V_ i = \eg^\top \bX^{(i)}/\sqrt{N}$ is also Gaussian with mean zero and covariance operator $C$.  Now apply Lemma \ref{l:n_exp_ineq} with $t = -\log(\delta)$ to obtain   
\begin{align*}
 P & \left\{   
\max_{1 \leq i \leq I} \|  V_i\|^2 \geq \| \Lambda\|_1  + 2 \|\Lambda\|_2 \sqrt{t + \log(I)}   + 2 \|\Lambda\|_\infty (t + \log(I))
\right\} \\
& \leq I \exp(-(t + \log(I)) = \exp(-t).
\end{align*}
Since $ \exp(-t) = \delta$, the claim holds.
\end{proof}
\end{lemma}

\begin{lemma} \label{l:bound}
On $\mscF$ with $\lambda \geq 2 \lambda_0$ we have that
\[
 \| \bX(\beta^\star-\widehat \beta)\|^2 + \lambda \| \widehat \beta_{S_0^c}\|_{\ell_1/\mcH}
\leq 3 \lambda  \| ( \widehat \beta - \beta^\star)_{S_0}\|_{\ell_1/\mcH}. 
\]
\begin{proof}
Applying Lemma \ref{l:basic_ineq} and multiplying both sides by 2, we have that, on $\mscF$,
\[
 \| \bX(\beta^\star-\widehat \beta)\|^2 + 2 \lambda \| \widehat \beta\|_{\ell_1/\mcH}
\leq \lambda  \|  \widehat \beta - \beta^\star\|_{\ell_1/\mcH}+ 2 \lambda \|  \beta^\star\|_{\ell_1/\mcH}. 
\]
Analyzing the penalty term on the LHS, we can apply the reverse triangle inequality to obtain
\[
\| \widehat \beta\|_{\ell_1/\mcH} = \| \widehat \beta_{S_0}\|_{\ell_1/\mcH} + \| \widehat \beta_{S_0^c} \|_{\ell_1/\mcH}
\geq   \|  \beta_{S_0}^\star \|_{\ell_1/\mcH} - \| \widehat \beta_{S_0} - \beta_{S_0}^\star \|_{\ell_1/\mcH} + \| \widehat \beta_{S_0^c} \|_{\ell_1/\mcH}
\]
and on the RHS we we have
\[
\| \widehat \beta - \beta^\star \|_{\ell_1/\mcH} 
= \| \widehat \beta_{S_0} - \beta_{S_0}^\star \|_{\ell_1/\mcH} +  \| \widehat \beta_{S_0^c}\|_{\ell_1/\mcH}.
\]
Combining these two calculations we arrive at the bound
\begin{align*}
 \| \bX(\beta^\star-\widehat \beta)\|^2 
 \leq & \lambda  \|  \widehat \beta - \beta^\star\|_{\ell_1/\mcH}+ 2 \lambda \|  \beta^\star\|_{\ell_1/\mcH}
-  2 \lambda \|  \widehat \beta \|_{\ell_1/\mcH} \\
 \leq & \lambda( \| \widehat \beta_{S_0} - \beta_{S_0}^\star \|_{\ell_1/\mcH} +  \| \widehat \beta_{S_0^c}\|_{\ell_1/\mcH}) 
+ 2 \lambda \|  \beta^\star_{S_0}\|_{\ell_1/\mcH}  \\
 & - 2 \lambda( \|  \beta_{S_0}^\star \|_{\ell_1/\mcH} - \| \widehat \beta_{S_0} - \beta_{S_0}^\star \|_{\ell_1/\mcH} + \| \widehat \beta_{S_0^c} \|_{\ell_1/\mcH}) \\
  = & 3 \lambda   \| \widehat \beta_{S_0} - \beta_{S_0}^\star \|_{\ell_1/\mcH} - \lambda  \| \widehat \beta_{S_0^c} \|_{\ell_1/\mcH},
\end{align*}
which is the desired result.
\end{proof}
\end{lemma}

\begin{proof}[Proof of Theorem \ref{t:highfreq_main}]
We examine both terms together to obtain
\begin{align*}
\| \bX(\widehat \beta-\beta^\star)\|^2 + \lambda \| \widehat \beta - \beta^\star\|_{\ell_1/\mcH} 
=  \|\bX(\widehat \beta-\beta^\star)\|^2 + \lambda \|( \widehat \beta - \beta^\star)_{S_0} \|_{\ell_1/\mcH} + \lambda \| \widehat \beta_{S_0^c} \|_{\ell_1/\mcH}.
\end{align*}
Applying Lemmas \ref{l:F} and \ref{l:bound} we have that, with probability $1-\delta$,
\[
\| \bX(\widehat \beta-\beta^\star)\|^2 + \lambda \|( \widehat \beta - \beta^\star)_{S_0} \|_{\ell_1/\mcH} + \lambda \| \widehat \beta_{S_0^c} \|
\leq 4 \lambda \|( \widehat \beta - \beta^\star)_{S_0} \|_{\ell_1/\mcH}\leq 4\lambda\sqrt{I_0}\|( \widehat \beta - \beta^\star)_{S_0} \|.
\]
Applying Definition \ref{d:f:comp}, we have that the RHS above is bounded by
\[
4 \lambda \sqrt{I_0}\|( \widehat \beta - \beta^\star)_{S_0} \|
\leq \frac{4 \lambda \sqrt{I_0}}{\sqrt{\alpha  N}} \| \bX(\beta^\star - \widehat \beta)\|
\leq \frac{1}{2} \| \bX(\beta^\star-\widehat \beta)\|^2 + \frac{8 \lambda^2 I_0}{\alpha N}  .
\]
Note the last inequality follows from the simple bound $4uv \leq \frac{1}{2}u^2+8v^2$ for any real $u$ and $v$ (the LHS just completes the square).  We therefore have that
\[
 \| \bX(\beta^\star-\widehat \beta)\|^2 + \lambda \|( \widehat \beta - \beta^\star)_{S_0} \|_{\ell_1/\mcH} + \lambda \| \widehat \beta_{S_0^c} \|_{\ell_1/\mcH}
\leq 
\frac{1}{2} \| \bX(\beta^\star-\widehat \beta)\|^2 + \frac{8 \lambda^2 I_0}{\alpha N} \;.
\]
This implies
\[
\frac{1}{2} \| \bX(\beta^\star-\widehat \beta)\|^2 + \lambda \| \widehat \beta - \beta^\star\|_{\ell_1/\mcH}  \leq \frac{8 \lambda^2 I_0}{\alpha N },
\]
which proves the claim.
\end{proof}

\end{document}